\documentclass[11pt]{article}
\usepackage{hyperref}
\usepackage{tikz}
\usetikzlibrary{decorations.pathreplacing}
\usepackage{benstyle}

\usepackage{enumitem}
\usepackage{comment}
\usepackage{cite}

\title{The devil in the (de)tails: an improved recovery guarantee for sparse approximation}
\author{Ben Adcock\thanks{Department of Mathematics, Simon Fraser University, Canada, \url{ben_adcock@sfu.ca}, \url{avi_gupta@sfu.ca}} \and Simone Brugiapaglia\thanks{Department of Mathematics and Statistics, Concordia University, Canada, \url{simone.brugiapaglia@concordia.ca}} \and Avi Gupta\footnotemark[1]}

\begin{document}

\maketitle

\begin{abstract}
Many functions exhibit approximate sparsity in their coefficients with respect to a given dictionary. In recent literature, sparse approximation in such a dictionary from i.i.d.\ pointwise samples, underpinned by compressed sensing, has become a powerful tool for high-dimensional function approximation. A key step in this framework is truncating the (typically countably-infinite) dictionary to a finite index set of size $n$, so that compressed sensing tools can be used to approximate the function by a sparse combination of these truncated dictionary elements. This introduces a discrete $L^2$-truncation error over the sample points, which in standard approaches, is bounded by the continuous $L^\infty$-norm. Such a deterministic, worst-case bound ignores the randomness of the sample points entirely. As a result, $n$ must be taken unnecessarily large to keep the truncation error under control, which directly inflates the size of the matrix involved in the sparse recovery algorithm and increases computational cost. In this paper, we show that by exploiting the i.i.d.\ structure of the sample points, the discrete $L^2$ truncation error admits a bound that instead reflects the faster decay behaviour of the continuous $L^2$-norm truncation error and yields significantly smaller truncation sets and decreased computational cost. We demonstrate this through applications to weighted Wiener spaces and anisotropic Sobolev spaces, in each case obtaining significantly smaller truncation sets than recent works. In addition, we also present an improved bound of independent interest for sparse approximation in bounded Riesz systems, where the measurement condition exhibits a smaller (and scale-invariant) dependence on the Riesz constants than in previous works.
\end{abstract}



\section{Introduction}\label{s:intro}

In recent years, sparse approximation in an orthonormal basis or dictionary has been established as an effective means to approximate functions from pointwise samples. Methods involving sparse approximation have proved effective in practice for, notably, high-dimensional approximation tasks \cite{rauhut2012sparse,rauhut2016interpolation,adcock2022sparse,adcock2024efficient}. More recently, sparse approximation has also been used effectively to establish new results on sampling numbers for various function spaces \cite{moeller2026best,moeller2023gelfand,moeller2025high,moeller2025sampling,moeller2024high,jahn2023sampling,adcock2024optimal,adcock2025optimal,adcock2024optimalb,adcock2026universal}, including classical Sobolev spaces, dominating mixed smoothness Sobolev spaces, mixed Wiener spaces and anisotropic spaces of infinite-dimensional
holomorphic functions. 

\subsection{Sparse approximation}

Let $(D,\cD,\rho)$ be a probability space, $L^2_{\rho}(D)$ be the space of square-integrable functions $f : D \rightarrow \bbC$ and $\{ \phi_i \}_{i \in \bbN}$ be a Riesz basis\footnote{Many of the aforementioned works consider orthonormal bases. A secondary contribution of this paper is to derive guarantees in the general setting of Riesz bases---see \S \ref{ss:discussion} for  further discussion. Since any orthonormal basis is a Riesz basis with $a_{\phi} = b_{\phi} =1$, all our results apply seamlessly to orthonormal bases as well.} of $L^2_{\rho}(D)$ with constants $a_{\phi},b_{\phi} > 0$, i.e., such that
\be{
\label{riesz-basis}
a_{\phi} \nm{c}^2_{2} \leq \nms{\sum_{i \in \bbN} c_i \phi_i }^2_{L^2_{\rho}} \leq b_{\phi} \nm{c}^2_{2},\quad \forall c \in \ell^2(\bbN).
}
We also assume that $\{ \phi_i \}_{i \in \bbN}$ is bounded, with constant $K_{\phi} < \infty$, i.e.,
\be{
\label{Kphi-def}
K_{\phi} = \sup_{i \in \bbN} \nm{\phi_i}_{L^{\infty}_{\rho}} < \infty.
}
The basic idea of sparse approximation is to approximate an unknown function $f$ from samples\footnote{Our extend readily to the noisy setting, since the SR-LASSO decoder we use (see \S \ref{ss:the-decoder}) is designed to handle samples corrupted by additive noise. Specifically, if \(y_i = f(x_i)+\eta_i\), \(i=1,\ldots,m\), the corresponding error bounds acquire an additional term proportional to \(\|\eta\|_2/{\sqrt{m}}\). We focus on noiseless samples \ef{f-samples} for simplicity.} 
\be{
\label{f-samples}
(x_1,f(x_1)),\ldots,(x_m,f(x_m))
}
by an $s$-sparse function in $\{ \phi_i \}_{i\in \mathbb{N}}$, 
i.e.,
\be{
\label{f-hat-sparse}
f \approx \hat{f} = \hat{c}_{i_1} \phi_{i_1} + \cdots + \hat{c}_{i_s} \phi_{i_s}.
}
This is typically done by recasting the problem as a sparse recovery problem for the coefficients of $f$. One first truncates the orthonormal basis using a finite, but large $n \geq m$, writes $f = \sum^{\infty}_{i=1} c_i \phi_i$ and then observes that the vector $c_{[n]} = (c_i)^{n}_{i=1}$ of the first $n$ coefficients satisfies
\be{
\label{CS-prob}
b = A c_{[n]} + e,
}
where
\be{
\label{bAe-def}
b = \frac{1}{\sqrt{m}} (f(x_i))^{n}_{i=1},\quad A = \frac{1}{\sqrt{m}} \left ( \phi_j(x_i) \right )^{m,n}_{i,j=1},\quad e = \frac{1}{\sqrt{m}} (f(x_i) - f_{n}(x_i))^{m}_{i=1}
}
and $f_n = \sum^{n}_{i=1} c_i \phi_i$ is the truncated expansion of $f$. One then seeks a sparse approximate solution of the linear system \ef{CS-prob}. This can be done by, for instance, solving a convex $\ell^1$-minimization problem, using a greedy algorithm such as Orthogonal Matching Pursuit (OMP) or by various other methods. For convenience, we shall write $\cR_{s,n}$ for the sparse recovery procedure that takes the samples \ef{f-samples} of $f$ and produces an approximation $\hat{f} = \cR_{s,n}(f)$. The specific procedure we use in this paper is defined in \S \ref{ss:the-decoder}.

Now consider the error $f - \hat{f}$. This can be written as
\bes{
f - \hat{f} = \underbrace{f_n - \hat{f}}_{\mathrm{(a)}} + \underbrace{f - f_n}_{\mathrm{(b)}}.
}
Here (a) is the \textit{sparse recovery error}, i.e., the error in recovering the approximately sparse vector of coefficients $c_{[n]}$ and (b) is the \textit{tail error}, i.e., the error due to truncation. In particular, $n$ should be chosen sufficiently large such that (b) is no larger than (a).

\subsection{The devil in the (de)tail}

Sparse approximation is undergirded by the theory of Compressed Sensing (CS) \cite{foucart2013mathematical}. A typical result takes the following form (for completeness, we prove this result in \S \ref{s:main-proofs}).

\thm{
[Standard CS bound]
\label{t:basic-CS-thm}
Let $0 < \varepsilon < 1$, $s \in \bbN$ and $x_1,\ldots,x_m \sim_{\mathrm{i.i.d.}} \rho$, where $m$ satisfies
\be{
\label{samp-rate}
m \geq c \cdot \left ( \frac{K^2_{\phi}}{a_{\phi}} \right ) \cdot s \cdot \left ( \log^2 \left (2 \frac{K^2_{\phi}}{a_{\phi}} s \right ) \log(2n) + \log(1/\varepsilon) \right )
}
for some universal constant $c > 0$.
Then the following holds with probability at least $1-\varepsilon$. For any $f \in L^2_{\rho}(D)$ that is defined everywhere, the approximation $\hat{f} = \cR_{s,n}(f)$ is at most $2s$-sparse and satisfies, for all $p \in [2,\infty]$,
\be{
\label{basic-CS-err-bd}
\nm{f - \hat{f}}_{L^p_{\rho}} \lesssim (b_{\phi})^{\frac1p} (K_{\phi})^{1-\frac2p} \left (  
\frac{\sigma_s(c_{[n]})_1}{s^{\frac1p}} +  \frac{s^{\frac12-\frac1p}}{\sqrt{a_{\phi}}} \nm{f - f_n }_{L^{\infty}_{\rho}} \right ).
}
}
Here and elsewhere, we write $a \lesssim b$ to mean $a \leq C b$ for some universal constant $C > 0$, $a \gtrsim b$ to mean $b \lesssim a$, and $a \asymp b$ to mean both $a \lesssim b$ and $a \gtrsim b$. When we write $\lesssim_{\lambda}$ the implicit constant may depend on the parameter(s) $\lambda$.

For succinctness, we do not specify how $\hat{f}$ is computed in Theorem \ref{t:basic-CS-thm}, nor in our main result below. However, it is based on solving an $\ell^1$-minimization problem followed by a thresholding step to obtain a $2s$-sparse approximation. See \S \ref{ss:the-decoder} for the full details  (note that one can obtain an $s$-sparse approximation subject to an additional term $s^{\frac12-\frac1p} \sigma_s(c_{[n]})_2$ in the error bound, as explained in Remark \ref{rem:2s-s-why}). We remark in passing that other decoders, such as OMP, Compressive Sampling Matching Pursuit (CoSaMP) or Hard Thresholding Pursuit (HTP) \cite{foucart2013mathematical}, could readily be considered, with minimal changes.

In Theorem \ref{t:basic-CS-thm}, the term $\sigma_s(c)_1$ is $\ell^1$-norm best $s$-term approximation error of $c$. In general, for $1 \leq p \leq \infty$, we define the $\ell^p$-norm best $s$-term approximation error of $c$ as
\bes{
\sigma_s(c)_p : = \inf \{ \nm{c-z}_{p} : \text{$z$ is $s$-sparse} \} \equiv \left ( \sum_{i > s} | c_{\pi(i)} |^p \right )^{1/p}
}
where $\pi : \bbN \rightarrow \bbN$ is a bijection that gives a nonincreasing rearrangement of $c = (c_i)_{i \in \bbN}$ by absolute value (also, ``$s$-sparse'' means that $z$ has at most $s$ nonzero entries).
Notice also that
\bes{
\nm{f - f_n}_{L^{\infty}_{\rho}} \leq \sqrt{K_{\phi}} \sum_{i > n} | c_i | = : \tau_n(c)_1,
}
where we likewise define the $\ell^p$-norm $n$-term approximation error of $c$ as
\bes{
\tau_n(c)_p = \inf \{ \nm{c - z}_p : \text{$z$ is nonzero in its first $n$ entries} \} \equiv \left ( \sum_{i > n} |c_i|^p \right )^{1/p}.
}
Hence, in the literature, one often encounters the following bound:
\be{
\label{bd-involving-tau}
\nm{f - \hat{f}}_{L^p_{\rho}} \lesssim  \frac{\sigma_s(c)_1}{s^{\frac1p}} + \sqrt{\frac{K_{\phi}}{a_{\phi}}} s^{\frac12-\frac1p} \tau_n(c)_1.
}
This result is broad and powerful, and has found used in the  majority of the aforementioned works. However, it suffers from a crucial limitation in that the truncation error $f - f_n$ is measured in the $L^{\infty}_{\rho}$ norm, while the overall error is measured in the $L^p_{\rho}$-norm.  As a result, when using Theorem \ref{t:basic-CS-thm} to derive some concrete rate of convergence in terms of $m$, as is usually the goal, one often needs to choose $n$ very (and, as we shall show, unnecessarily) large in comparison to $s$. While this only has a logarithmic effect on the sampling rate in view of \ef{samp-rate}, it also has practical consequences: the system \ef{CS-prob} is of size $m \times n$ and the vast majority of sparse recovery algorithms will therefore incur a computational cost scaling linearly in $n$.

\subsection{Main contribution}

The main contribution of this paper is a new bound which improves Theorem \ref{t:basic-CS-thm} by replacing the term $ \nm{f - f_n }_{L^{\infty}_{\rho}}$ by error terms measured in the $L^2_{\rho}$-norm only. Specifically:

\thm{
[Main result]
\label{t:main-res}
Let $0 < \varepsilon < 1$, $s \in \bbN$ and $x_1,\ldots,x_m \sim_{\mathrm{i.i.d.}} \rho$, where $m$ satisfies
\be{
\label{samp-cond-2}
m \geq c \cdot \left ( \frac{K^2_{\phi}}{a_{\phi}} \right ) \cdot s \cdot \left ( \log^2 \left (2 \frac{K^2_{\phi}}{a_{\phi}} s \right ) \log(2n) + \log(2/\varepsilon) \right )
}
for some universal constant $c > 0$.
Then the following holds with probability at least $1-\varepsilon$.  For any $f \in L^2_{\rho}(D)$ that is defined everywhere, the approximation $\hat{f} = \cR_{s,n}(f)$ is at most $2s$-sparse and satisfies, for all $0 < u ,v < 2$,
\be{
\label{main-err-bd-l2}
\nm{f - \hat{f}}_{L^2_{\rho}} \lesssim_{u,v} \sqrt{b_{\phi}} \Bigg\{ \frac{\sigma_s(c_{[n]})_1 }{\sqrt{s}}
+  \frac{\sqrt{b_{\phi}}}{\sqrt{a_{\phi}}} \left [ \left ( \frac1s \sum_{j > s} \sigma_j(c)^u_2 \right )^{\frac1u} + \left ( \frac1n \sum_{j > n/2} \tau_j(c)^v_2 \right )^{\frac1v} \right ] \Bigg \}
}
and, for all $p \in (2,\infty]$,
\be{
\label{main-err-bd-lp}
\begin{split}
\nm{f - \hat{f}}_{L^{p}_{\rho}}  \lesssim_{u,v} (b_{\phi})^{1/p} (K_{\phi})^{1-2/p} \Bigg \{ &
\frac{\sigma_s(c_{[n]})_1 }{s^{\frac1p}}
+ \frac{\tau_n(c)_1}{s^{\frac1p}}  
\\
& +  \frac{s^{\frac12-\frac1p}\sqrt{b_{\phi}}}{\sqrt{a_{\phi}}} \left [ \left ( \frac1s \sum_{j > s} \sigma_j(c)^u_2 \right )^{\frac1u} + \left ( \frac1n \sum_{j > n/2} \tau_j(c)^v_2 \right )^{\frac1v} \right ] \Bigg \}.
\end{split}
}
}

Note that $\hat{f}$ is computed in exactly the same way as in Theorem \ref{t:basic-CS-thm}. See \S \ref{ss:the-decoder} for details.
Theorem \ref{t:main-res} replaces the term $\tau_n(c)_1$ appearing in \ef{bd-involving-tau} by two infinite sequences, involving the best $j$-term approximation error $\sigma_j(c)_2$ and the $j$-term approximation $\tau_j(c)_2$. Crucially, both are now with respect to the $\ell^2$-norm (which is equivalent to the $L^2_{\rho}$-norm).

We now discuss two examples that highlight the main benefit of this result: namely, it allows one to choose a much smaller $n$ than that that would arise from Theorem \ref{t:basic-CS-thm}.

\subsection{Example: weighted mixed Wiener spaces}\label{ss:weighted-wiener-intro}

As a first example, we consider the weighted Wiener spaces, which have been studied in various recent works \cite{moeller2026best,kolomoitsev2023sparse,nguyen2022s,moeller2023gelfand,moeller2025high,moeller2025sampling,jahn2023sampling,krieg2025sampling,moeller2024high}. Let $\mathbb{T}^d = [0,1)^d$ the $d$-dimensional torus, with the $L^2(\bbT^d)$-orthonormal Fourier basis
\begin{align}
\psi_k(x) = \E^{2\pi \I k \cdot x}, \qquad k \in \mathbb{Z}^d.
\label{psi-def}
\end{align}
For $f \in L^2(\bbT^d)$, we $f = \sum_{k \in \bbZ^d} \hat{f}_k \psi_k$ (with convergence in $L^2$), where its Fourier coefficients are given by
\begin{align}
\hat{f}_k = \int_{\bbT^d} f(x) \psi_{-k}(x)\, \D x, \qquad k \in \bbZ^d .
\label{fhat-moeller-def}
\end{align}

\begin{definition}[Weighted mixed Wiener space]
For $r > 0$ and $\theta \in (0,\infty]$ the \textit{weighted mixed Wiener space} $S^r_{\theta} \cA(\bbT^d) = S^r_{\theta} \cA : = \{ f \in L_1(\mathbb{T}^d) : \|f\|_{S^r_{\theta} \cA} < \infty \}$, where
\begin{align*}
\|f\|_{S^r_{\theta} \cA} = \begin{cases} \left( \displaystyle\sum_{k \in \mathbb{Z}^d} \prod_{i=1}^d (1+|k_i|)^{r \theta} |\hat{f}(k)|^{\theta} \right)^{1/\theta} & \theta < \infty, \\ \displaystyle\sup_{k \in \mathbb{Z}^d} \prod_{i=1}^d (1+|k_i|)^{r} |\hat{f}(k)| & \theta = \infty. \end{cases}
\end{align*}
\end{definition}

Since the Fourier basis is indexed over $\bbZ^d$, the truncated expansion $f_n$ of $f$ takes the form $f_n = \sum_{k \in \Lambda} \hat{f}_k \phi_k$ for some index set $\Lambda \subset \bbZ^d$ of size $|\Lambda| = n$. The following theorem shows there exists a good choice of $\Lambda$ for this problem, where $n$ is not too large. Here and elsewhere, we use the notation $(x)_+ = \max \{ x , 0 \}$ for $x \in \bbR$.

\thm{
\label{t:wiener-main}
Let $\theta \in (0,\infty]$, $r > (1-1/\theta)_+$ and $s \in \bbN$. Then there is a choice of $\Lambda \subseteq \bbZ^d$ with
\bes{
| \Lambda | = n = \begin{cases} \lceil s^{(r+1/\theta-1/2)/(r - (1-1/\theta)_+ ) } \rceil & r \leq 1/2
\\
\max \{ s , \lceil s^{(r+1/\theta-1/2)/r} \rceil \} & r > 1/2
\end{cases},
}
such that the following holds. Let $\varepsilon \in (0,1)$ and $x_1,\ldots,x_m \sim_{\mathrm{i.i.d.}} \rho$, where $\rho$ is the uniform measure on $\bbT^d$,
and $m$ satisfies
\bes{
m \geq c_{d,r,\theta} \cdot s \cdot \left ( \log^3(2s) + \log(2/\varepsilon) \right ),
}
where $c_{d,r,\theta} > 0$ depends on $d$, $r$, and $\theta$ only.
Then the following holds with probability at least $1-\varepsilon$. For any $f \in S^{r}_{\theta}\cA$ the approximation $\hat{f} = \cR_{s,n}(f)$ is at most $2s$-sparse and satisfies
\be{
\label{wiener-rate}
\nm{f - \hat{f}}_{L^p} \lesssim_{r,\theta} s^{1-1/p-1/\theta-r} \log^{(d-1)r}(s+1).
}
}

Note that $r > (1-1/\theta)_+$ is a necessary and sufficient condition for $S^{r}_{\theta}(\bbT^d) \hookrightarrow C(\bbT^d)$,
which ensures that pointwise evaluations of $f$ are well-defined, allowing for uniform recovery bounds that hold simultaneously for all $f \in S^r_{\theta} \cA$ given a single draw of the sample points. 

As we discuss further in \S \ref{ss:wiener-samp-widths}, the rate \ef{wiener-rate} is nearly optimal: it leads to an upper bound for the \textit{sampling width} $\varrho_m(S^{r}_{\theta} \cA)_{L^2}$ that is within a polylogarithmic factor of known lower bounds. Indeed, a known lower bound (see \cite[Lem.\ B.1]{jahn2023sampling}) for $\theta \in (0,2]$ is
\bes{
\varrho_m(S^{r}_{\theta} \cA)_{L^2} \gtrsim_{r,\theta} m^{1/2-1/\theta-r} \log^{(d-1) r}(m+1),
}
while our result gives the upper bound
\be{
\label{rho-m-upper-bound}
\varrho_m(S^{r}_{\theta} \cA)_{L^p} \lesssim_{r,\theta} m^{1-1/p-1/\theta-r} \log^{(d-1)r + 3(r+1/\theta+1/p-1)}(m+1),\quad \forall p \in [2,\infty].
}
As noted, many recent works have considered the spaces $S^{r}_{\theta} \cA$, with an algorithm achieving the upper bound \ef{rho-m-upper-bound} being recently introduced in \cite{moeller2026best}. This algorithm is based on a similar sparse recovery problem. The primary improvement of Theorem \ref{t:wiener-main} is the size $n$ of the truncation set $\Lambda$. As we discuss in \S \ref{ss:moller-comparison}, in \cite[Cor.\ 6.2]{moeller2026best}, the truncation set $\Lambda$ is chosen as
\bes{
\Lambda = [-M,M]^d \cap \bbZ^d,\quad \text{where } M \asymp s^{(r+1/\theta-1/2) / (r - (1-1/\theta)_{+} ) },
}
and therefore
\bes{
n = |\Lambda| \gtrsim s^{d(r+1/\theta-1/2) / (r - (1-1/\theta)_{+} ) }.
}
Theorem \ref{t:wiener-main} shows the same rate can be achieved with a truncation set that is at least a power of $d$ smaller. In particular, $n$ does not suffer from the curse of dimensionality. 

\subsection{Example: universal algorithms for anisotropic dominating mixed smoothness spaces}

In our second example, we let $\alpha = (\alpha_1,\ldots,\alpha_d) > 0$ be an \textit{anisotropy parameter} and consider the \textit{anisotropic dominating mixed smoothness Sobolev spaces}
\be{
\label{Halphamix}
H^{\alpha}_{\mathsf{mix}}(\bbT^d) = \left\{f\in L^2(\bbT^d): \|f\|^2_{H^\alpha_{\mathsf{mix}}} := \sum_{k \in \mathbb{Z}^d} \prod_{j=1}^d (1 + |k_j|)^{2\alpha_j} |\hat{f}_n|^2<\infty\right\}.
}
These spaces have been studied extensively \cite{temlyakov2018multivariate,dung2018hyperbolic,novak2008trac}. Sampling recovery and sampling widths of these function spaces, both linear and nonlinear, have attracted significant recent attention \cite{krieg2021function,krieg2021function2,nagel2022new,kammerer2021worst,dolbeault2023sharp,jahn2023sampling,dai2024random,kosov2025sampling,moeller2025high,moeller2026best,adcock2026universal,byrenheid2017optimal}.
Our focus is on the development of \textit{universal algorithms}, meaning algorithms that achieve near-optimal recovery rates \textit{simultaneously} for all possible values of $\alpha$, without a priori knowledge of the function's smoothness. Algorithms of this type were recently introduced in \cite{adcock2026universal}. We improve upon that work by obtaining significantly smaller truncation sets. Like \cite{adcock2026universal} our algorithm is nonadaptive, in contrast to adaptive approaches such as \cite{bartel2025learning}, which iteratively estimate the anisotropy from the data, and \cite{binev2005universal,binev2007universal}, which adaptively partition the domain.

For convenience, we now define
\be{
\label{h-p-def}
h(\alpha)=\min_{\substack{i\in[d]}}\alpha_i,\quad p(\alpha)=|\{j\in[d]:\alpha_j=h(\alpha)\}|,\qquad \alpha \in [0,\infty)^d,
}
We assume that $\alpha > 1/2$ (understood componentwise), as this ensures that $H^{\alpha}_{\mathsf{mix}}(\bbT^d) \hookrightarrow C(\bbT^d)$.

\thm{
\label{t:aniso-sob-main}
There there are constants $C(\alpha,d) > 0$ for all $\alpha > 1/2$ and, for each $s \in \bbN$, a choice of $\Lambda \subseteq \bbZ^d$ with 
\be{
\label{aniso-sob-n-new}
n = |\Lambda| \leq s \log^{d-1}(\E s),
}
such that the following holds. Let $\varepsilon \in (0,1)$ and $x_1,\ldots,x_m \sim_{\mathrm{i.i.d.}} \rho$, where $\rho$ is the uniform measure on $\bbT$ and $m$ satisfies
\be{
\label{m-cond-HC}
m \geq c_{d} \cdot s \cdot \left ( \log^3(2s) + \log(2/\varepsilon) \right ),
}
where $c_{d} > 0$ depends on $d$ only.
Then the following holds with probability at least $1-\varepsilon$. For any $\alpha > 1/2$ and $f \in H^{\alpha}_{\mathsf{mix}}(\bbT^d)$ the approximation $\hat{f} = \cR_{s,n}(f)$ is at most $2s$-sparse and satisfies
\be{
\label{Halpha-mix-rate}
\nm{f - \hat{f}}_{L^2} \lesssim_{d} \left ( \frac{\log^{p(\alpha)-1}(s)}{s} \right )^{h(\alpha)} \nm{f}_{H^{\alpha}_{\mathsf{mix}}}.
}
}

Note that the rate \ef{Halpha-mix-rate} is optimal and the sparse recovery procedure $\cR_{s,n}$ is independent of $\alpha$ (since $\Lambda$ is independent of $\alpha$). This result improves that of \cite{adcock2026universal} this result by significantly reducing the size of the truncation set and simplifying the overall construction. As we discuss in \S \ref{ss:Halphamix-comparison}, the index set $\Lambda$ used in \cite[Thm.\ 3.1]{adcock2026universal}  satisfies
\bes{
n = |\Lambda| \gtrsim_{d} s^{u(s)} (u(s) \log(s+1))^{d-1},
}
where $u(s)$ is some fixed, but arbitrary, increasing function of $s$ with $u(s) \rightarrow \infty$ as $s \rightarrow \infty$. In particular, $n$ grows superalgebraically with $s$ as $s \rightarrow \infty$.
By contrast, \ef{aniso-sob-n-new} is much smaller, i.e., linear in $s$ up to the log term.
Another by-product of this new bound is we also slightly reduce the measurement condition. In \cite[Thm.\ 3.1]{adcock2026universal} it is
\bes{
m \geq c_d \cdot s \cdot \left ( \log^3(2s) \cdot u(s) + \log(1/\varepsilon) \right ).
}
By using a smaller index set, we eliminate the factor $u(s)$.

\subsection{Additional contributions}

As noted, sparse approximation has been used notably for function approximation tasks and, in particular, to establish optimal algorithms for recovery in various function spaces. See \cite{adcock2024optimal,adcock2024optimalb,adcock2025optimal,adcock2026universal,jahn2023sampling,moeller2023gelfand,moeller2024high,moeller2025high,moeller2025sampling,moeller2026best} and references therein. Our work contributes in this direction by improving the error bound for sparse approximation, yielding, as discussed more efficient algorithms in terms of computational cost. Besides, we anticipate out main Theorem \ref{t:main-res} to be of independent theoretical interest, as it develops an error bound using only $\ell^2$-norm quantities.

While this is our main contribution, we also make three further contributions that are of independent interest.

\pbk
\textbf{(a) Improved recovery guarantees for bounded Riesz bases.} Sparse approximation in bounded Riesz bases was considered in \cite{brugiapaglia2021sparse} (see also \cite{jung2022estimation}). Our proofs use several results from this work to establish Theorem \ref{t:main-res}, notably, certain concentration and deviation bounds \cite[Thms.\ 1.1 \& 4.2]{brugiapaglia2021sparse} and \cite[Thm.\ 1.34]{jung2022estimation}. The main results in \cite{brugiapaglia2021sparse} are comparable to Theorem \ref{t:basic-CS-thm} in that they bound the truncation error using $\tau_n(c)_1$. However, we improve on \cite{brugiapaglia2021sparse} by establishing a sharper measurement condition. Specifically, \cite[Thm.\ 2.6]{brugiapaglia2021sparse}, when translated into our notation, gives a measurement condition of the form
\be{
\label{SB-comp}
m \geq c \cdot \left ( \frac{b_{\phi}}{a_{\phi}} \right )^2 \cdot K^2_{\phi} \cdot s \cdot  \log^2 \left (2 \frac{K^2_{\phi} b_{\phi} }{ a_{\phi}} s \right ) \cdot \log(2n) 
}
whenever $b_{\phi} \geq 1$ 
(note that the failure probability is treated slightly differently in \cite[Thm.\ 2.6]{brugiapaglia2021sparse}, so we ignore it for the purposes of comparison). In \ef{samp-rate} and \ef{samp-cond-2} we improve this to
\be{
\label{SB-comp-ours}
m \geq c \cdot \left( \frac{ K^2_{\phi} }{a_{\phi}} \right ) \cdot s \cdot  \log^2 \left (2 \frac{K^2_{\phi}}{ a_{\phi}} s \right ) \cdot \log(2n), 
}
thus reducing the dependence on the Riesz constants $a_{\phi},b_{\phi}$. In particular, the upper Riesz constant $b_{\phi}$ does not appear in our measurement condition, and the dependence on $a_{\phi}$ is lessened. It is notable that our result is invariant to scaling, i.e., the operation $\phi_i \rightarrow \gamma \phi_i$ leaves \ef{SB-comp-ours} unchanged. This is not the case for \ef{SB-comp}, which would increase by at least $\gamma^2$.

\pbk
\textbf{(b) Sparse approximations.} The solution of an $\ell^1$-minimization is generally not sparse. Hence, procedures based on $\ell^1$-minimization do not generically produce sparse approximations. A secondary contribution of this paper is to show that one can always postprocess the 
output of an $\ell^1$-minimization program by hard thresholding to obtain a genuinely sparse approximation. In Lemma~\ref{lem:threshold}, which is based on \cite[Cor.\ 3.2]{needell2010signal}, we show that hard thresholding an approximate solution does not substantially worsen the error, and the resulting sparse approximation satisfies the same error bounds up to numerical constants.

\pbk
\textbf{(c) General weighted Wiener-type spaces.} Our results in \S \ref{ss:weighted-wiener-intro} consider weighted mixed Wiener spaces $S^r_{\theta}\cA(\bbT^d)$, which have been studied in many recent works. However, our results in \S \ref{s:weighted-wiener} apply to general weighted Wiener-type spaces defined by arbitrary orthonormal or Riesz bases (i.e., not just the Fourier basis) and with general weights. In Theorem \ref{t:abstract-wiener-main} we present a general result for weights of the form $w_i = i^{r} / \log^t(i+1)$, from which Theorem \ref{t:wiener-main} follows as a special case. Thus are results are substantially more general than recent works \cite{moeller2026best,kolomoitsev2023sparse,nguyen2022s,moeller2023gelfand,moeller2025high,moeller2025sampling,jahn2023sampling,krieg2025sampling,moeller2024high}, which are specific to the spaces $S^r_{\theta}\cA(\bbT^d)$.

\subsection{Further discussion and related work}\label{ss:discussion}

Our work is inspired by Krieg \& Ullrich's contributions \cite{krieg2021function,krieg2021function2}. In a pair of breakthrough papers they showed that function values are essentially as powerful as arbitrary linear information for $L^2$ recovery, by proving that the sampling numbers are upper bounded by tail averages of the approximation numbers of the embedding. This resolved the open question of whether function values achieve the same rate of convergence as optimal linear information for $L^2$ recovery. Subsequent works \cite{nagel2022new,kammerer2021worst,dolbeault2023sharp} built upon and refined these results, proving similar bounds with improved logarithmic factors. 
These works use (weighted) linear least-squares estimators as their recovery algorithm, and thus are applicable to function classes for which linear algorithms are optimal. 
Our work uses nonlinear sparse approximation, and is applicable to spaces such as the mixed Wiener spaces and universal recovery in the anisotropic Sobolev spaces, where nonlinear algorithms are required in order to achieve optimal rates. 

More concretely, and using our notation, the results of \cite{krieg2021function,krieg2021function2} establish error bounds involving the $j$-term approximation errors $\tau_j(c)$ only: namely,
\be{
\label{krieg-ullrich-bound}
\nm{f - \hat{f} }_{L^2_{\rho}} \lesssim_{v} \sqrt{\frac{b_{\phi}}{a_{\phi}}} \left ( \frac1n \sum_{j > n/2} \tau_j(c)^v_2 \right )^{\frac1v} 
}
where $\hat{f}$ is a linear least-squares estimator, subject a number of i.i.d.\ samples scaling log-linearly in $n$, i.e.,
\bes{
m \geq c \cdot \left ( \frac{K^2_{\phi}}{a_{\phi}} \right ) \cdot n \cdot  \log \left (2n/\varepsilon  \right ).
}
(note that \cite{krieg2021function,krieg2021function2} also show the existence of a set of sample points $x_1,\ldots,x_m$ for $m \geq c n$ samples suffice). Our result involves a measurement condition \ef{samp-cond-2} that is linear in the sparsity $s$, and only logarithmic in $n$, with additional terms involving the best $s$- and $j$-term approximation errors $\sigma_s(c)_1$ and $\sigma_j(c)_2$, $j > s$.

The results in \cite{krieg2021function,krieg2021function2} are obtained by (i) carefully decomposing the tail term $e$ in \ef{bAe-def} into dyadic sums, and (ii) making repeated use of Chernoff bounds to estimate the norms of the ensuing matrices. In combination with a careful balancing of parameters, this leads to the desired error bound \ef{krieg-ullrich-bound}. Our approach is similar, although substantially more involved as we deal nonlinear best $s$-term approximation errors $\sigma_s(c)_p$ in addition to the linear $n$-term approximation errors $\tau_n(c)_p$. Instead of (i), we partition the tail term into a carefully-constructed countable collection of sparse index sets lying within certain dyadic ranges. And instead of (ii) we use sophisticated deviation bounds (whose proofs are based on advanced chaining techniques) to upper bound the action of the resulting matrices on arbitrary sparse vectors. See \S \ref{ss:our-technique} for an overview of our proof.

As noted, our work is in part motivated by computational considerations. Notably, iterative algorithms for sparse approximation incur a computational cost of the form $\ord{T m n}$, where $T$ is the number of iterations. In OMP, for instance, $T = \ord{s}$. In this work, we consider the SR-LASSO optimization program. This was introduced in statistics in \cite{belloni2011square-root} and first used in the context of function approximation in \cite{adcock2019correcting}. See also \cite{adcock2024efficient,moeller2025high,moeller2025sampling}.  As shown in \cite{adcock2024efficient}, it can be solved efficiently using the primal-dual iteration \cite{chambolle2011first,chambolle2016ergodic} combined with a restart scheme \cite{adcock2025restarts}. In this case, $T = \ord{\log(1/\varepsilon)}$ iterations suffice to achieve an error within $\varepsilon$ of the exact minimizer.
Crucially, in all such approaches, the truncation set size $n$ directly determines the size of the matrix $A$ and hence the computational cost, which is one of the main motivations for keeping $n$ small. We remark in passing that \textit{sublinear time} algorithms \cite{choi2021sparse1,choi2021sparse2} can avoid incurring a cost scaling with $n$. However, these use specialized sample points, and to date, do not attain optimal error bounds under a sampling budget scaling log-linearly with the sparsity $s$.

\subsection{Outline}

The outline of the remainder of this paper is as follows. In \S \ref{s:key-CS-tools}, we introduce some key compressed sensing tools that are used in our analysis. In \S \ref{s:main-proofs} we establish the main result, Theorem \ref{t:main-res}. In \S \ref{s:weighted-wiener} and \S \ref{s:anisotropic-sobolev} we consider its application to weighted Wiener-type spaces and anisotropic Sobolev spaces, respectively. Finally, we close with a conclusion in \S \ref{s:conclusion}.

\section{Key compressed sensing tools}\label{s:key-CS-tools}

We now introduce some key compressed sensing tools. Here and elsewhere we use the notation $[n] : = \{1,\ldots,n\}$ for $n \in \bbN$ and $[n_1,n_2] := \{n_1,\ldots,n_2\}$ for $n_1,n_2 \in \bbN$, $n_1 < n_2$.

\subsection{The sparse approximation procedure}\label{ss:the-decoder}

We first specify how the sparse approximation $\hat{f}$ is computed. As noted, there are many ways to do this. Following \cite{belloni2011square-root}, we will employ the so-called \textit{Square-Root LASSO} decoder. Given a regularization parameter $\lambda > 0$, a matrix $A \in \mathbb{C}^{m \times N}$ and vector $b \in \mathbb{C}^m$, the \textit{(unconstrained) Square Root (SR)-LASSO problem} is the optimization problem
\be{
\label{eq:srlasso}
\min_{z \in \bbC^n} \lambda \nm{z}_1 + \nm{A z - b}_2,
}
The advantage of this decoder over, say, the classical LASSO is that a theoretically-optimal choice of $\lambda$ is independent of the noise term $e$ in \ef{CS-prob} which, in our setting, depends on the unknown expansion tail $f-f_n$.

The program \ef{eq:srlasso} does not generally yield sparse solutions. To obtain a $2s$-sparse approximation, we hard threshold the coefficients. Let $H_{2s} : \bbC^n \rightarrow \bbC^n$ be the hard-thresholding operator, i.e., for $z = (z_i)^{n}_{i=1} \in \bbC^n$, $H_{2s}(z)$ is the $2s$-sparse vector with $i$th entry $z_i$ if $|z_i|$ is one of the $2s$ largest entries of $z$ in absolute value and zero otherwise. A secondary contribution of this work shows that solutions of \ef{eq:srlasso} that are postprocessed by $H_{2s}$ still satisfy the same error bounds, up to constants.

Another potential complication is that \ef{eq:srlasso} generally has infinitely-many solutions. This is easily dealt with by picking one. In practice, this could be the output of some optimization algorithm for solving \ef{eq:srlasso}. But, theoretically, arguably the simplest choice is to pick the solution with the minimal $\ell^2$-norm (which is unique, as the $\ell^2$-norm is strictly convex and the set of minimizers of \ef{eq:srlasso} is a nonempty, closed and convex set).

\pbk
\textbf{The sparse approximation procedure in Theorems \ref{t:basic-CS-thm} and \ref{t:main-res}.}  We now specify this procedure. Given samples \ef{f-samples} of a function $f$ and $A,b$ as in \ef{CS-prob}, we define $\hat{f} = \sum^{N}_{i=1} \hat{c}_i \phi_i$, where $\hat{c} = H_{2s}(\check{c})$ and
\bes{
\check{c} = \argmin{} \left \{ \nm{\tilde{c}}_{2} : \text{$\tilde{c}$ is a minimizer of \ef{eq:srlasso} with } \lambda = \frac{3}{14} \sqrt{\frac{a_{\phi} }{ s}} \right \}.
}
Note that the choice of $\lambda$ is based on the theory we develop later in this section.

\rem{[Why $2s$ and not $s$]
\label{rem:2s-s-why}
The choice of $2s$ leads to a slightly more appealing error bound. As can be seen in the proofs of the main results in \S \ref{s:main-proofs}, if one were to use $s$-sparse approximation $\hat{c} = H_s(\check{c})$ then the error bounds in Theorems \ref{t:basic-CS-thm} and \ref{t:main-res} would involve an additional term of the form $s^{\frac12-\frac1p}\sigma_s(c)_2$.
}

\subsection{Preliminaries}

We now recap some preliminary compressed sensing concepts. See, e.g., \cite[Defn.\ 5.14, Lem.\ 5.15 \& 5.16]{adcock2021compressive}.

\defn{[$\ell^2$-rNSP]
Given $s\in \mathbb{N}$, a matrix $A\in\mathbb{C}^{m\times N}$ satisfies the \textit{$\ell^2$-robust Null Space Property ($\ell^2$-rNSP) of order $s$ with constants $\rho \in (0,1)$ and $\tau > 0$}
if, for all $z\in \mathbb{C}^N$ and all index sets $S\subseteq[N]$ with $|S|=s$,
\begin{align}\label{l2r}
\|z_S\|_2 \leq \frac{\rho}{\sqrt{s}} \|z_{S^c}\|_1 + \tau \|A z\|_2.
\end{align}
}

\lem{[rNSP implies stable and accurate recovery for the SR-LASSO problem]\label{l:srlasso_implies_stable_and_accurate_recovery}
Let $A \in \mathbb{C}^{m \times N}$ satisfy the rNSP of order $s$ with constants $0 < \rho < 1$ and $\tau > 0$. 
Let $x \in \mathbb{C}^N$, $h \in \mathbb{C}^m$, and $b = A x + e \in \mathbb{C}^m$. 
Then, for any $\lambda$ satisfying $\lambda \in \left(0,\frac{D}{\sqrt s}\right]$ where
$
D = \frac{(1+\rho)}{(3+\rho)\tau}$,
and any $\hat{x}$ such that
\bes{
\hat{x} \in \argmin{z \in \mathbb{C}^N}~\lambda \|z\|_1 + \|A z - b\|_2,
}
we have
\begin{align*}
\|x - \hat{x} \|_1 
&\le C_1 \sigma_s(x)_1 
+ \frac{1}{2}\left( \frac{C_1}{\lambda} + C_2\sqrt{s} \right)\|e\|_2,
\quad
\|x - \hat{x}\|_2 
&\le C_3 \frac{\sigma_s(x)_1}{\sqrt{s}} 
+ \frac{1}{2}\left( \frac{C_3}{\sqrt{s}\lambda} + C_4 \right)\|e\|_2,
\end{align*}
where the constants $C_1, C_2, C_3, C_4$ are given by
\bes{
C_1 = 2\left( \frac{1+\rho}{1-\rho} \right), \quad
C_2 = \frac{4 \tau}{1-\rho}, \quad
C_3 = \frac{2(1 + \rho)^2}{(1 - \rho)}, \quad
C_4 = 2\tau \frac{(3 + \rho)}{(1 - \rho)}.
}
}

Next, we require the following result, which shows that postprocessing an approximation $\check{c}$ to a vector $c$ by hard thresholding does not substantially increase the error. This is based on \cite[Cor.\ 3.2]{needell2010signal}. We include a short proof for completeness.

\lem{
\label{lem:threshold}
Let $c,\check{c} \in \bbC^n$ and $\hat{c} = H_s(\check{c})$. Then, for any $1 \leq p \leq \infty$,
\bes{
\nm{c - \hat{c}}_p \leq 3 \nm{c - \check{c}}_p + 3 \sigma_s(c)_p.
}
}
\prf{
Let $S,T \subseteq [n]$, $|S|,|T| \leq s$ be such that $\hat{c} = \check{c}_S$ and $H_s(c) = c_T$. In particular, $\sigma_s(c)_p = \nm{c - c_T }_p = \nm{c_{T^c}}_p$. Then
\eas{
\nm{c - \hat{c}}_p = \nm{c - \check{c}_S}_{p} \leq \nm{ c_T - \hat{c}_S}_p + \nm{c_{T^c}}_p
& \leq  \nm{( c_T - \check{c}_S )_S}_p  +  \nm{c_{T \backslash S}}_p + \sigma_s(c)_p
 \\
 & = \nm{(c_T - \check{c})_S}_p +  \nm{c_{T \backslash S}}_p + \sigma_s(c)_p
 \\
 & \leq \nm{c_T - \check{c}}_p + \nm{c_{T \backslash S}}_p + \sigma_s(c)_p
 \\
 & \leq \nm{c - \check{c} }_p + \nm{c_{T \backslash S}}_p + 2\sigma_s(c)_p.
}
Now consider the second term. We have
\bes{
\nm{c_{T \backslash S}}_p  \leq \nm{(c-\check{c})_{T \backslash S}}_p + \nm{\check{c}_{T \backslash S}}_p \leq \nm{c - \check{c}}_p + \nm{\check{c}_{T \backslash S}}_p.
}
Observe that $|S| = |T| = s$ and therefore $|S \backslash T| = | T \backslash S |$. Since $S$ contains the largest $s$ entries of $\check{c}$ in absolute value, we must have $\nm{\check{c}_{T \backslash S}}_p \leq \nm{\check{c}_{S \backslash T}}_p$.
We deduce that
\bes{
\nm{\check{c}_{T \backslash S}}_p \leq \nm{\check{c}_{S \backslash T}}_p = \nm{(\check{c}-c_T)_{S \backslash T}}_p \leq \nm{\check{c} - c_T}_{p} \leq \nm{\check{c} -c }_2 + \sigma_s(c)_p.
}
Hence $\nm{c_{T \backslash S}}_p  \leq 2 \nm{c - \check{c}}_p + \sigma_s(c)_p$.
We now combine this with the above inequality.
}

\subsection{The rNSP for bounded Riesz systems}

We now consider when the rNSP holds for random sampling with bounded Riesz systems. As in \S \ref{s:intro}, we now let $(D,\cD,\rho)$ be a probability space and $L^2_{\rho}(D)$ be the space of square-integrable functions $f : D \rightarrow \bbC$.  

\thm{[rNSP for random sampling in bounded Riesz systems]
\label{t:bounded-riesz-rNSP}
There exists universal constants $c,c' >0$ such that the following holds. Let $0 < \delta , \varepsilon < 1$, $n \in \bbN$ and $\{ \phi_i \}^{n}_{i =1} \subset L^2_{\rho}(D) \cap L^{\infty}_{\rho}(D)$ be linearly independent, and define
\bes{
a_{\phi} = \inf_{\substack{c \in \bbC^n \\ \nm{c}_2 = 1}} \nms{\sum^{n}_{i=1} c_i \phi_i}_{L^2_{\rho}} > 0,\qquad K_{\phi} = \max_{i =1,\ldots,n} \nm{\phi_i}_{L^{\infty}_{\rho}} < \infty.
}
Let $x_1,\ldots,x_m \sim_{\mathrm{i.i.d.}} \rho$ and consider the matrix $A = \frac{1}{\sqrt{m}} \left ( \phi_j(x_i) \right )^{m,n}_{i,j=1} \in \bbC^{m \times n}$, 
where $m$ satisfies
\be{
\label{m-cond-riesz-rNSP}
m \geq c \frac{(1+1/\rho)^2}{(1-1/\tau^2)^{2}} \frac{K^2_{\phi}}{a_{\phi}}  s \left [ \log(2 N) \log^2 \left ( \frac{ 2  (1+1/\rho)^2}{(1-1/\tau^2)} \frac{K^2_{\phi}}{a_{\phi}} s \right ) \log^2 \left (\frac{2 }{(1-1/\tau^2)} \right ) + \log(1/\varepsilon) \right ]
}
for some $0 < \rho < 1$ and $\tau > 1$. Then, with probability at least $1-\varepsilon$, $A$ has the rNSP of order $s$ with constants $\rho$ and $\tau/\sqrt{a_{\phi}}$. In particular, $A$ has the rNSP of order $s$ with constants $\rho = 1/2$ and $\tau = 2 / \sqrt{a_{\phi}}$ with probability at least $1-\varepsilon$, provided
\bes{
m \geq c' \frac{K^2_{\phi}}{a_\phi} s \left [ \log(2 n) \log^2 \left ( \frac{ 2  K^2_{\phi}}{a_{\phi}} s \right ) + \log(1/\varepsilon) \right ].
}
}

To prove this theorem, we require the following two results.

\thm{
\label{t:simone-1}
There exist absolute constants $c_0,c_1,c_2 >0$ and $\kappa \in (0,1)$ such that the following holds. Let $X_1,\ldots,X_m$ be independent copies of a random vector $X \in \bbC^N$ with bounded coordinates, i.e., for all $i = 1,\ldots,N$ we have $| \ip{X}{e_i}| \leq K$ for some $K > 0$ where $e_1,\ldots,e_N$ is the standard basis of $\bbC^N$. Let $T \subseteq \sqrt{s} B^{N}_{1}$, where $B^N_1 = \{ x \in \bbC^N : \nm{x}_1 \leq 1 \}$, $\delta \in (0,\kappa)$ and assume that
\bes{
m \geq c_0 K^2 \delta^{-2} s \log(\E N) \log^2(s K^2/\delta) \log^2(1/\delta).
}
Then with probability exceeding $1-2 \exp(-c_1 \delta^2 m / (s K^2))$,
\bes{
\sup_{f \in T}  \left | \frac1m\sum^{m}_{i=1} | \ip{f}{X_i}|^2 - \bbE | \ip{f}{X} |^2 \right | \leq c_2 \left ( \delta + \delta \sup_{f \in T} \bbE | \ip{f}{X} |^2 \right ).
}
}

See \cite[Thm.\ 1.34]{jung2022estimation}.\footnote{This result first appeared in \cite[Thm.\ 1.1]{brugiapaglia2021sparse} with an incorrect dependence on $\delta$ in the main condition on $m$, before being subsequently corrected in \cite{jung2022estimation}. We also amend Theorem \ref{t:simone-2} in the same way.}
We also require the following lemma, which can be found in \cite{brugiapaglia2021sparse} (we give a short proof for completeness). 

\lem{
Let $1 \leq s \leq N$, $0 < \rho < 1$, $A \in \bbC^{m \times N}$, $B \in \bbC^{N \times N}$ be nonsingular and
\be{
\label{T-def}
T = \left \{ x \in \bbC^N : \exists S \subseteq [N],\ |S| \leq s,\ \nm{x_S}_{2} \geq \frac{\rho}{\sqrt{s}} \nm{x_{S^c}}_{1} \right \}.
}
Suppose that
\bes{
\inf_{\substack{x \in T \\ \nm{B x}_2 = 1}} \nm{A x}_{2} \geq 1/\tau,
}
for some $\tau > 0$. Then $A$ has the rNSP of order $s$ with constants $\rho$ and $\tau/s_N(B)$, where $s_N(B) > 0$ is the minimum singular value of $B$.
} 
\prf{
Let $x \in \bbC^N$ and suppose first that $x \in T$. Then $x / \nm{B x}_2 \in T$ and therefore we have
\bes{
s_{N}(B) \nm{x}_2 \leq \nm{B x}_2 \leq \tau \nm{A x}_2.
}
Let $S \subseteq [N]$, $|S| = s$. Then
\bes{
\nm{x_S}_2 \leq \nm{x}_2 \leq \frac{\tau}{s_N(B)} \nm{A x}_2 \leq \frac{\rho}{\sqrt{s}} \nm{x_{S^c}}_1 + \frac{\tau}{s_N(B)} \nm{A x}_2.
}
Hence $x$ satisfies the desired condition for the rNSP. Suppose next that $x \notin T$. Then, for any $S \subseteq [N]$, $|S| = s$, we have
\bes{
\nm{x_S}_2 < \frac{\rho}{\sqrt{s}} \nm{x_{S^c}}_1  \leq \frac{\rho}{\sqrt{s}} \nm{x_{S^c}}_1 + \frac{\tau}{s_N(B)} \nm{A x}_2.
}
Hence $x$ also satisfies the desired condition. The result follows.
}

\prf{[Proof of Theorem \ref{t:bounded-riesz-rNSP}]
The matrix $A$ satisfies
\be{
\label{AstarA-gram}
\bbE(A^*A) = G : =\left ( \ip{\phi_i}{\phi_j}_{L^2_{\rho}} \right )^{N}_{i,j=1} \in \bbC^{N \times N}.
}
The matrix $G$ is the Gram matrix of the first $N$ basis functions, and is positive definite due to linear independence. Let $B$ be its unique positive definite square-root and notice that $s_N(B) = \sqrt{\lambda_N(G)} \geq \sqrt{a_{\phi}}$, where $\lambda_N(G)$ denotes the $N$th (and smallest) eigenvalue of $G$ and $a_{\phi}$ is as in \ef{riesz-basis}. 
By the previous lemma, we want to show that
\bes{
I := \inf_{\substack{x \in T \\ \nm{B x}_2 = 1}} \nm{A x}_{2} \geq 1/\tau
}
with probability at least $1-\varepsilon$, where $T$ is as in \ef{T-def}. By \ef{AstarA-gram} and the definition of $B$, we have $\bbE \nm{A x}^2_2 = \nm{B x}^2_2$. Hence
\bes{
I^2 \geq 1 - \sup_{\substack{x \in T \\ \nm{B x}_2 = 1}} \left | \nm{A x}^2_{2} - \bbE \nm{A x}^2_2 \right | =: 1 - J.
}
To estimate $J$, we aim to use Theorem \ref{t:simone-1}. Define the random vector $X = (\phi_i(x))^{m}_{i=1} \in \bbC^m$, where $x \sim \rho$. Notice that
\be{
\label{A-relate-to-X}
\nm{A x}^2_{2} - \bbE \nm{A x}^2_2 = \frac1m \sum^{m}_{i=1} | \ip{x}{X_i} |^2 - \bbE | \ip{x}{X} |^2
}
and that 
\be{
\label{X-inf-norm}
\nm{X}_{\infty} \leq \max_{i \in [m]} \nm{\phi_i}_{L^{\infty}_{\rho}} \leq K_{\phi}.
}
Now let $x \in T$ with $\nm{B x}_2 = 1$. Then there is a set $S \subseteq [N]$, $|S| \leq s$ such that $\nm{x_S}_2 \geq \frac{\rho}{\sqrt{s}} \nm{x_{S^c}}_1$. Hence
\eas{
\nm{x}_1  = \nm{x_S}_1 + \nm{x_{S^c}}_1 
\leq \sqrt{s} \nm{x_S}_2 + \frac{\sqrt{s}}{\rho} \nm{x_S}_2 
\leq \sqrt{s} \left ( 1 + \frac{1}{\rho} \right ) \nm{x}_2
 \leq \frac{\sqrt{s}}{\sqrt{a_{\phi}}}\left ( 1 + \frac{1}{\rho} \right )
}
We deduce that
\bes{
T \cap \{ x : \nm{B x}_2 = 1 \} \subseteq \frac{\sqrt{s}}{\sqrt{a_{\phi}}}\left ( 1 + \frac{1}{\rho} \right ) B^{N}_1.
}
We now apply Theorem \ref{t:simone-1} with $K = K_{\phi}$, $s$ replaced by $\frac{s}{a_{\phi}} \left ( 1 + \frac{1}{\rho} \right )^2$ and $T$ replaced by $T \cap \{ x : \nm{B x}_2 = 1 \}$. This asserts that if
\bes{
m \geq c_0 K^2_{\phi} \delta^{-2} a^{-1}_{\phi} \left ( 1 + \frac{1}{\rho} \right )^2 s \log(\E N) \log^2 \left ( \frac{s K^2_{\phi} \left ( 1 + \frac{1}{\rho} \right )^2 }{a_{\phi} \delta} \right ) \log^2(1/\delta)
}
Then 
\bes{
J \leq c_2 \delta \left ( 1 + \sup_{\substack{x \in T \\ \nm{B x}_2 = 1}} \bbE \nm{A x}^2_2 \right ) \leq 2 c_2 \delta
}
with probability at least $1-p$, where
\bes{
p = 2 \exp \left ( - \frac{c_1 a_{\phi} \delta^2 m}{s (1+1/\rho)^2 K^2_{\phi} } \right ).
}
We deduce that
$I^2 \geq 1 - 2 c_2 \delta$
with the same probability. Hence $I \geq 1/\tau$ with the same probability, provided $
\delta \leq \frac{1-1/\tau^2}{2 c_2}$.
Without loss of generality, we may assume that $\frac{1}{2 c_2} < \kappa$. Hence, we now set $\delta = \frac{1-1/\tau^2}{2 c_2 }$. It follows that $I \geq 1/\tau$ with probability at least $1-p$, where
\bes{
p = 2 \exp \left ( - \frac{c'_1 a_{\phi} m (1-1/\tau^2)^2}{s(1+1/\rho)^2  K^2_{\phi}} \right ),
}
provided
\bes{
m \geq c'_0\frac{(1+1/\rho)^2}{(1-1/\tau^2)^{2}}   \frac{K^2_{\phi}}{a_{\phi}} s \log(\E N) \log^2 \left (  \frac{ 2 c_2 (1+1/\rho)^2}{(1-1/\tau^2)} \frac{K^2_{\phi}}{a_{\phi}} s \right ) \log^2 \left (\frac{2 c_2}{(1-1/\tau^2)} \right ).
}
Now observe that $K_{\phi} = \sup_{i \in \bbN} \nm{\phi_i}_{L^{\infty}_{\rho}} \geq \sup_{i \in \bbN} \nm{\phi_i}_{L^2_{\rho}} \geq \sqrt{a_{\phi}}$, since $\rho$ is a probability measure and $\{ \phi_i \}_{i \in \bbN}$ is a Riesz basis. Therefore, the condition on $m$ can be replaced by
\bes{
m \geq c \frac{(1+1/\rho)^2}{(1-1/\tau^2)^{2}} \frac{K^2_{\phi}}{a_{\phi}}  s \log(2 N) \log^2 \left ( \frac{ 2  (1+1/\rho)^2}{(1-1/\tau^2)} \frac{K^2_{\phi}}{a_{\phi}} s \right ) \log^2 \left (\frac{2 }{(1-1/\tau^2)} \right ).
}
This condition is implied by \ef{m-cond-riesz-rNSP}. Hence we deduce that $A$ has the rNSP with constant $\rho$ and $\tau/\sqrt{a_{\phi}}$, with probability at least $1-p$, where $p$ is as above. However, \ef{m-cond-riesz-rNSP} also implies that $p \leq \varepsilon$. The result now follows.
}

\subsection{A deviation bound for bounded Riesz systems}

Another component of our analysis is the following estimate, that bounds how much $\nm{A x}^2_2$ can exceed its mean for sparse vectors. Notice that this result does not place a condition on $m$ such as \ef{m-cond-riesz-rNSP}. This is crucial in our later estimates.

\thm{
\label{t:bounded-riesz-deviation}
There exist universal constants $c_1,c_2,c_3>0$ with $c_2\leq1$ such that the following holds. Let $\{ \phi_i \}^{n}_{i=1}$, $a_{\phi}$, $K_{\phi}$ and $A$ be as in Theorem \ref{t:bounded-riesz-rNSP}.
Let $1 \leq s \leq n$, $m \in \bbN$, $g > 0$, $B$ be the unique positive definite square-root of the matrix $\bbE(A^*A)$ and $T = \{ x \in \bbC^n : \nm{x}_0 \leq s,\ \nm{B x}_2 = 1 \}$. Suppose that
\be{
\label{Exp-cond}
\frac{s  K^2_{\phi} \log^2(2 c_1K^2_{\phi} s/a_\phi ) \log(2 n) }{a_{\phi} m} + 1  \leq c_2 g.
}
Then
\bes{
\bbP \left ( \sup_{x \in T} \nm{A x}_2 \geq \sqrt{g+1} \right ) \leq 2 \exp \left ( - c_3 g \frac{a_{\phi} m}{K^2_{\phi} s }\right )
}
}

To prove the above theorem, we require the following result (see \cite[Thm.\ 4.2]{brugiapaglia2021sparse}).

\thm{
\label{t:simone-2}
There exist absolute constants $c_1,c_2,c_3>0$ such that the following holds. Let $X,X_1,\ldots,X_m$ be as in Theorem \ref{t:simone-1}, $T \subseteq \sqrt{s} B^N_{1}$, and $\delta \in (0,1)$. Then
\eas{
\bbE & \sup_{f \in T}  \left | \frac1m\sum^{m}_{i=1} | \ip{f}{X_i}|^2 - \bbE | \ip{f}{X} |^2 \right | 
\\
& \leq c_1 \sqrt{\frac{s K^2 \log^2(s K^2 / \delta) \log(\E N) \log^2(1/\delta)}{m} } \sqrt{ \bbE \sup_{f \in T} \frac1m  \sum^{m}_{i=1} | \ip{f}{X_i}|^2 } 
\\
& ~~~ + c_2 \delta\ \bbE \sup_{f \in T} \frac1m \sum^{m}_{i=1} | \ip{f}{X_i} |^2 + c_3 \delta.
}
}

\prf{[Proof of Theorem \ref{t:bounded-riesz-deviation}]
Define the random variable $Z = \sup_{x \in T} \left | \nm{A x}^2_2 - \nm{B x}^2_2 \right |$.
We first bound $\bbE(Z)$.
As in the previous proof, let  $X = (\phi_i(x))^{m}_{i=1}$ for $x \sim \rho$ and $X_1,\ldots,X_m$ be independent copies of $X$. Recall that \ef{A-relate-to-X} and \ef{X-inf-norm} hold and observe that $T \subseteq \sqrt{s/a_{\phi}} B^n_{1}$. 
Hence Theorem \ref{t:simone-2} gives that
\bes{
\bbE(Z) \leq C_1 \sqrt{\frac{s K^2_{\phi} \log^2\left(s K^2_{\phi}/(a_{\phi}\delta)\right) \log^2(1/\delta) \log(\E n) }{a_{\phi} m} } \sqrt{\bbE \sup_{x \in T} \nm{A x}^2_2 } + C_2 \delta \bbE \sup_{x \in T} \nm{A x }^2_2 + C_3 \delta
}
for any $\delta \in (0,1)$ and some universal constants $C_1,C_2,C_3>0$.
Observe that $\bbE \sup_{x \in T} \nm{A x}^2_2 \leq \bbE(Z) + 1$.
Hence 
\bes{
\bbE(Z) \leq C_1 \sqrt{\frac{s K^2_{\phi} \log^2\left(s K^2_{\phi}/(a_{\phi}\delta)\right) \log^2(1/\delta) \log(\E n) }{a_{\phi} m} } \left ( \sqrt{\bbE(Z)} + 1 \right ) + C_2 \delta (\bbE(Z)+1) + C_3 \delta.
}
Assuming without loss of generality that $C_2 \geq 1$, we now pick $\delta = 1/(2 C_2)$ to  obtain
\bes{
\bbE(Z) \leq C_4 \sqrt{\frac{s K^2_{\phi} \log^2\left(2C_2s K^2_{\phi}/a_{\phi}\right) \log(\E n) }{a_{\phi} m} } \left ( \sqrt{\bbE(Z)} + 1 \right ) + C_5,
}
for some universal constants $C_4, C_5 > 0$.
This is a quadratic inequality for $\bbE(Z)$. Completing the square, we deduce that
\be{
\label{EZ-bd}
\bbE(Z) \leq c \left ( \frac{s K^2_{\phi} \log^2\left(2C_2s K^2_{\phi}/a_{\phi}\right) \log(\E n) }{a_{\phi} m} + 1 \right ),
}
for some universal constant $c>0$.
Having done this, we now look to bound $Z$ in probability. For this we use Talagrand's concentration inequality. See, e.g., \cite[Thm.\ 4.1]{brugiapaglia2021sparse}. For this example, we have $\cF = \{ | \ip{x}{\cdot} |^2 / m : x \in T \}$, $\sigma^2_{\cF} \leq K^2_{\phi} s / (a_{\phi} m)$ and $\beta_{\cF} \leq K^2_{\phi} s / (a_{\phi} m )$. Therefore
\bes{
\bbP \left (Z \geq \bbE(Z) + \sqrt{2 \frac{u K^2_{\phi} s}{a_{\phi} m} ( 1 + 2 \bbE(Z)) } + \frac13 \frac{K^2 s }{a_{\phi} m} u \right ) \leq 2 \exp(-u),\quad \forall u > 0.
}
Now, let $u_0 > 0$ be the unique solution of $\sqrt{2 u_0 K^2_{\phi} s / (a_{\phi} m) (1 + 2\bbE(Z)))} + K^2_{\phi} s u_0/ (3a_{\phi} m)  = g/2$. Then $\bbP(Z \geq \bbE(Z) + g/2) \leq 2\exp(-u_0)$ and we have
\bes{\label{eq:u0-lower-bound}
u_0\ge \frac{g^2/4}{2K^2_{\phi} s / (a_{\phi} m)(1+2\bbE(Z)) + K^2_{\phi} s g/ (3a_{\phi} m)}.
}
Further, we have $Z \geq \sup_{x \in T} \nm{A x}^2_2 - 1$.
Hence
\bes{
\bbP \left ( \sup_{x \in T} \nm{A x}_2 \geq \sqrt{g+1} \right ) \leq \bbP \left ( Z \geq g \right )
}
Now suppose that $\bbE(Z) \leq g/2$,
which, due to \ef{EZ-bd}, is implied by \ef{Exp-cond} with $c_1=C_2$ and $c_2\le\min\{1,1/2c\}$. Moreover, since the left-hand side of \ef{Exp-cond} is at least $1$, we have $c_2g \geq 1$. Hence $c_2 \leq  1$ implies $g \geq  1$. This along with $\bbE(Z) \leq g/2$ implies that  $1 + 2\bbE(Z) \leq 1 + g \leq 2g$. Hence, by \eqref{eq:u0-lower-bound}, $u_0$ satisfies
\begin{align*}
    u_0\geq \frac{3g}{52} \cdot\frac{a_{\phi} m}{K^2_{\phi} s}.
\end{align*}
Then, for an appropriate constant $c_3>0$, we have 
\bes{
\bbP \left ( \sup_{x \in T} \nm{A x}_2 \geq \sqrt{g+1} \right ) \leq \bbP \left ( Z \geq g \right ) \leq \bbP \left ( Z \geq \bbE(Z) + \frac{g}{2} \right ) \leq 2 \exp \left ( - c_3 g \frac{a_{\phi} m}{K^2_{\phi} s }\right ),
}
as required.
}

\section{Proof of the main result}\label{s:main-proofs}

In this section, we prove the main result of the paper, Theorem \ref{t:main-res}. 

\subsection{Proof of Theorem \ref{t:basic-CS-thm}}

For completeness, we also give a short proof of Theorem \ref{t:basic-CS-thm} as it will inform the proof of Theorem \ref{t:main-res} later.

\prf{[Proof of Theorem \ref{t:basic-CS-thm}]
Let $E$ be the event that $A$ has the rNSP of order $s$ with constants $\rho = 1/2$ and $\tau = 2 / \sqrt{a_{\phi}}$. If $E$ occurs, then Lemma \ref{l:srlasso_implies_stable_and_accurate_recovery} gives that
\be{
\label{step-bounds}
\| c_{[n]} - \check{c} \|_1 
\lesssim   \sigma_s(c_{[n]})_1 
+  \frac{\sqrt{s}}{\sqrt{a_{\phi}}} \|e\|_2 ,
\quad
\| c_{[n]} - \check{c} \|_2 
\lesssim  \frac{\sigma_s(c_{[n]})_1}{\sqrt{s}} 
+  \frac{1}{\sqrt{a_{\phi}}} \|e\|_2 ,
}
where $e$ is as in \ef{bAe-def}. We now apply Lemma \ref{lem:threshold} with $p=1,2$ to deduce that
\be{
\label{l1-bd-temp}
\nm{c_{[n]}-\hat{c}}_1 \lesssim  \sigma_s(c_{[n]})_1 
+  \frac{\sqrt{s}}{\sqrt{a_{\phi}}} \|e\|_2 
}
(here we used the fact that $\sigma_{2s}(c_{[n]})_1 \leq \sigma_s(c_{[n]})_1$)
and
\be{
\label{l2-bd-tmp}
\| c_{[n]} - \hat{c} \|_2 
\lesssim  \frac{\sigma_s(c_{[n]})_1}{\sqrt{s}} 
+  \frac{1}{\sqrt{a_{\phi}}} \|e\|_2  + \sigma_{2s}(c_{[n]})_2.
}
We wish to bound the term $\sigma_{2s}(c_{[n]})_2$. Let $z \in \ell^1(\bbN)$ and $z^*$ be its best $s$-term approximation. Then, by Stechkin's inequality  (see, e.g., \cite[Lem.~3.5]{adcock2022sparse}), 
\bes{
\sigma_{2s}(z)_2 = \sigma_{s}(z-z^*)_2 \lesssim \frac{\nm{z-z^*}_1}{\sqrt{s}} = \frac{\sigma_{s}(z)_1}{\sqrt{s}}. 
}
Using this, we get
\be{
\label{l2-bd-tmp2}
\| c_{[n]} - \hat{c} \|_2 
\lesssim \frac{\sigma_s(c_{[n]})_1}{\sqrt{s}} 
+  \frac{1}{\sqrt{a_{\phi}}} \|e\|_2 .
}
We now derive bounds for $f - \hat{f}$ in the $L^2_{\rho}$ and $L^{\infty}_{\rho}$-norms.
Using \ef{riesz-basis},  \ef{l2-bd-tmp2} and the fact that $\nm{e}_2 \leq \nm{f - f_n}_{L^{\infty}_{\rho}}$, we have
\bes{
\nm{f - \hat{f}}_{L^2_{\rho}} \leq \sqrt{b_{\phi}} \nm{c_{[n]} - \hat{c}}_2 + \nm{f - f_n}_{L^{\infty}_{\rho}} \lesssim \sqrt{b_{\phi}} \frac{\sigma_{s}(c_{[n]})_1}{\sqrt{s}} + \sqrt{\frac{b_{\phi}}{a_\phi}} \nm{f - f_n}_{L^{\infty}_{\rho}}.
}
Using \ef{Kphi-def} and \ef{l1-bd-temp} we have 
\bes{
\nm{f - \hat{f}}_{L^{\infty}_{\rho}} \leq K_{\phi} \nm{c_{[n]} - \hat{c}}_1 + \nm{f - f_n}_{L^{\infty}_{\rho}} \lesssim K_{\phi} \sigma_s(c_{[n]})_1 + \frac{K_{\phi}}{\sqrt{a_{\phi}} } \sqrt{s} \nm{f - f_n}_{L^{\infty}_{\rho}}.
}
Note that here we also used the fact that $K_{\phi} \geq \sqrt{a_{\phi}}$, which follows from \ef{Kphi-def}, \ef{riesz-basis} and the fact that $\rho$ is a probability measure. 
We now use the following interpolation inequality, which is a standard result (see, e.g., \cite[Lem.~A.1.7]{temlyakov2018multivariate}) 
\be{
\label{Lp-prod-split}
\nm{g}_{L^p_{\rho}} \leq \nm{g}^{1-2/p}_{L^{\infty}_{\rho}} \nm{g}^{2/p}_{L^2_{\rho}},\quad \forall g \in L^{\infty}_{\rho}(D)
}
to obtain the desired error bound \ef{basic-CS-err-bd} for arbitrary $p \in [1,\infty]$.

Therefore, it remains to show that $\bbP(E) \geq 1-\varepsilon$. However, this follows immediately from \ef{samp-rate}, Theorem \ref{t:bounded-riesz-rNSP} and the fact that $K^2_{\phi} / a_{\phi} \geq 1$.
}

\subsection{Overview of the proof of Theorem \ref{t:main-res}}\label{ss:our-technique}

The rest of this section proves Theorem \ref{t:main-res}. Before doing so, we first give a brief overview of the argument.

The first step is to pass from the infinite tail $f - f_n$ to a sum of finite pieces. This precise construction is obtained in Lemmas \ref{l:partition}-\ref{lem:Tail-bound}. Specifically, we first divide the indices $\{n+1,n+2,\ldots \}$ into pieces $\{ n_{k-1}+1,\ldots,n_k \}$, $k = 1,2,\ldots$, where $n_0 = n$. Then, given sparsity parameters $0 = s_0 < s_1 < s_2 < \cdots$, we partition $\{n+1,n+2,\ldots\}$ into index sets $T_1, T_2, \ldots$ depending on the coefficient vector $c$, where $T_k \subseteq \{ n_{k-1}+1,\ldots,n_k\}$ satisfies $|T_k| \leq s_k$. These index sets have two key properties. First, the norm of the corresponding coefficients it controllable: namely,
\bes{
\nm{c_{T_k}}_2 \leq \sigma_{s_{k-1}}(c)_2 + \tau_{n_{k-1}}(c)_2.
}
Second, each $T_k$ have controlled cardinality specified by $s_k$, and it lies within a finite range specified by $n_k$. This partition is visualized in Figure \ref{fig:partition}.

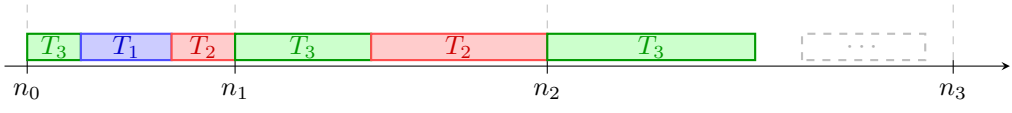
\begin{figure}[t]
\begin{center}
\begin{tikzpicture}[scale=1.0, >=stealth]
  \def\bh{0.35}
  \def\by{0.08}

  \draw[->] (-0.3,0) -- (13.0,0);

  \foreach \x/\lbl in {0/n_0, 2.75/n_1, 6.88/n_2, 12.25/n_3} {
    \draw (\x,0.08) -- (\x,-0.1) node[below] {\small$\lbl$};
    \draw[dashed, gray!50] (\x,\by) -- (\x,\by+\bh+0.5);
  }

  \draw[fill=green!20, draw=green!60!black, thick] (0,\by) rectangle (0.71,\by+\bh);
  \node[green!60!black, font=\small] at (0.36,\by+\bh/2) {$T_3$};

  \draw[fill=blue!20, draw=blue!70, thick] (0.71,\by) rectangle (1.91,\by+\bh);
  \node[blue!80!black, font=\small] at (1.31,\by+\bh/2) {$T_1$};

  \draw[fill=red!20, draw=red!70, thick] (1.91,\by) rectangle (2.75,\by+\bh);
  \node[red!80!black, font=\small] at (2.33,\by+\bh/2) {$T_2$};

  \draw[fill=green!20, draw=green!60!black, thick] (2.75,\by) rectangle (4.55,\by+\bh);
  \node[green!60!black, font=\small] at (3.65,\by+\bh/2) {$T_3$};

  \draw[fill=red!20, draw=red!70, thick] (4.55,\by) rectangle (6.88,\by+\bh);
  \node[red!80!black, font=\small] at (5.71,\by+\bh/2) {$T_2$};

  \draw[fill=green!20, draw=green!60!black, thick] (6.88,\by) rectangle (9.63,\by+\bh);
  \node[green!60!black, font=\small] at (8.26,\by+\bh/2) {$T_3$};

  \draw[draw=gray!50, dashed, thick] (10.25,\by) rectangle (11.88,\by+\bh);
  \node[gray!60, font=\small] at (11.07,\by+\bh/2) {$\cdots$};

\end{tikzpicture}
\end{center}
\caption{The partition constructed in Lemma \ref{l:partition}.}
\label{fig:partition}
\end{figure}

The construction of this partition allows us to bound the term $e$ in \ef{bAe-def} as $\nm{e}_2 \leq \sum_{k=1}^\infty \nm{A_k c_{T_k}}_2$, where $A_k$ is an $m \times (n_k - n_{k-1})$ matrix. Since our results are uniform guarantees (i.e., they hold simultaneously for all functions), we need to bound $\nm{A_k c_{T_k}}_2$ in terms of $\nm{c_{T_k}}_2$ for \textit{all} possible index sets $T_k$ of the given size. We do this via the deviation inequality, Theorem \ref{t:bounded-riesz-deviation}, yielding
\bes{
\nm{e}_2 \leq \sum^{\infty}_{k=1} \theta_k \left ( \sigma_{s_{k-1}}(c)_2 + \tau_{n_{k-1}}(c)_2 \right )
}
for constants $\theta_k$ depending on $n_k$, $s_k$, $m$ and $\varepsilon$. The final step is to choose the parameters $n_k$ and $s_k$ suitably. We make the dyadic choices $s_k = 2^k s$ and $n_k = 2^k n$, which fixes the ratio $s_k/n_k$ across all levels, and then perform a series of algebraic manipulations to obtain the desired error bound.

\subsection{Proof of Theorem \ref{t:main-res}}

We first present the two lemmas that form the crux of our analysis.

\lem{
\label{l:partition}
Let $z \in \ell^1(\bbN)$ and consider integers $1 = n_0 < n_1 < n_2 < \cdots$ and $0 = s_0 < s_1 < s_2 < \cdots$. Let $t_k = s_{k} - s_{k-1}$, $k \in \bbN$, and define $T_1$ as the index set of the largest $t_1$ entries of $z$ in absolute value in $\{n_0+1,\ldots,n_1\}$, with the assumption that $T_1 = \{n_0+1,\ldots,n_1\}$ if $t_1 \geq n_1 - n_0+1$. Then define $T_2$ as the index set of the largest $t_2$ entries of $z$ in absolute value in $\{n_0+1,\ldots,n_2\} \backslash T_1$, $T_3$ as the index set of the largest $t_3$ entries of   $z$ in absolute value in $\{n_0+1,\ldots,n_3\} \backslash (T_1 \cup T_2)$ and so forth. Then the collection $T_1,T_2,\ldots$ forms a partition of
\bes{
\mathrm{supp}(z) = \{ i : z_i \neq 0 \} \subseteq \bbN.
}
}
\prf{
Clearly the sets are disjoint. Hence we only need to show that their union is $\mathrm{supp}(z)$. Suppose first that $r : = |\mathrm{supp}(z)|  < \infty$. Then $\mathrm{supp}(z) \subseteq \{n_0+1,\ldots,n_k \}$ for some $k$. Since $|T_i| = t_i \geq 1$, it follows that $T_1 \cup \cdots \cup T_{k+r} \supseteq \mathrm{supp}(z)$, as required.

Now suppose that $|\mathrm{supp}(z)| = + \infty$. Consider an index $i \in \mathrm{supp}(z)$. Define
\bes{
S_i = \{ j : |z_j| \geq |z_i| \},\quad m_i = |S_i|,\quad a_i = \max S_i.
}
Notice that $m_i , a_i < \infty$ since $z \in \ell^1(\bbN)$. Now choose $k$ sufficiently large so that $n_k \geq a_i$ and $s_k > m_i$. This is possible, since $\{n_i\}$ and $\{s_i\}$ are strictly increasing sequences of positive integers. We now show that $i \in T_1 \cup \cdots \cup T_k$, by arguing by contradiction. Suppose that $i \notin T_1 \cup \cdots \cup T_k$. Then every $j \in T_1 \cup \cdots \cup T_k$ is such that $|z_j| \geq |z_i|$. Hence $T_1 \cup \cdots \cup T_k \subseteq S_i$.
However, by construction,
$
|T_1 \cup \cdots \cup T_k| = t_1 + \cdots + t_k = s_k - s_0 = s_k,
$
while $|S_i| = m_i < s_k$. This is a contradiction. Therefore $i \in T_1 \cup \cdots \cup T_k$, as required.
}

\lem{
[Tail bound]\label{lem:Tail-bound}
Suppose that $f = \sum_{i \in \bbN} c_i \phi_i \in L^2_{\rho}(D)$ be defined everywhere and let $e$ be as in \ef{bAe-def}. Consider integers $n = n_0 < n_1 < n_2 < \cdots$ and $0 = s_0 < s_1 < s_2 < \cdots$. Then
\bes{
\nm{e}_2 \leq \sum^{\infty}_{k=1} \theta_k \left ( \sigma_{s_{k-1}}(c)_2 + \nm{c - c_{[n_{k-1}]}}_2 \right ),
}
where
\bes{
\theta_k = \sup \left \{ \nm{A_k z}_2 : z \in \bbC^{n_k}, \nm{z}_0 \leq s_k, \nm{z}_2 =1 \right \},\quad 
A_k = \frac{1}{\sqrt{m}} \left ( \phi_j(x_i) \right )_{i \in [m],j \in [n_0+1, n_k] }.
}
}
\prf{
Let $z = (c_i)^{\infty}_{i=n+1}$. We first use Lemma \ref{l:partition} to construct a partition $T_1,T_2,\ldots$ of $\mathrm{supp}(z) = \mathrm{supp}(c) \backslash [n]$. 
Using this, we can write $f - f_n = \sum^{\infty}_{k=1} \sum_{j \in T_k} c_j \phi_j$.
Since $|T_k| = t_k \leq s_k$, this implies that
\bes{
\nm{e}_2 \leq \sum^{\infty}_{k=1} \theta_k \nm{c_{T_k}}_2.
}
Now consider the terms $\nm{c_{T_k}}_2$. For $k = 1$, since $T_1 \subseteq \{n_0+1,\ldots,n_1\}$, where $n_0 = n$, 
we have
\bes{
\nm{c_{T_1}}_2 \leq \nm{c - c_{[n]}}_2 = \nm{f - f_n}_{L^2_{\rho}}.
}
Next, consider $k = 2$. We have
\eas{
\nm{c_{T_2}}^2_2 &\leq \sum_{i \in [n_0+1,n_2] \backslash T_1 } |c_i|^2 = \sum_{i \in [n_0+1,n_1] \backslash T_1} |c_i|^2 + \sum_{i \in [n_1+1,n_2]} |c_i |^2
\leq \sigma_{t_1}(c_{[n_0+1,n_1]})^2_2 + \nm{c - c_{[n_1]} }^2_2.
}
We now claim that $\sigma_s(c_{\Lambda})_p \leq \sigma_s(c)_p$ for any $\Lambda \subseteq \bbN$. Indeed, let $S$ be the index set of the largest $s$ entries of $c$ in absolute value and $S'$ be the index set of the largest $s$ entries of $c$ in $\Lambda$ in absolute value. Then
\bes{
\sigma_s(c_{\Lambda})_1 = \nm{c} - \nm{c_{\Lambda^c}}_1 - \nm{c_{S'}}_1 = \sigma_s(c)_1 + \nm{c_S}_1 - \nm{c_{\Lambda^c}}_1 - \nm{c_{S'}}_1 = \sigma_s(c)_1 + \nm{c_S}_1 - \nm{c_{S' \cup \Lambda^c}}_1.
}
Now $S$ must be a subset of $S' \cup \Lambda^c$. Therefore $\nm{c_{S' \cup \Lambda^c}}_1 \geq \nm{c_S}_1$. The claim now follows.

Applying this claim and recalling that $t_1 = s_1 - s_0 = s_1$, we see that
\bes{
\nm{c_{T_2}}^2_2 \leq \sigma_{s_1}(c)^2_2 + \nm{c-c_{[n_1]}}^2_2.
}
Now consider general $k \geq 2$. We write 
\eas{
\nm{c_{T_k}}^2_2  \leq \sum_{i \in [n_0,\ldots,n_k] \backslash (T_1 \cup \cdots \cup T_{k-1}) } |c_i|^2
& \leq \sum_{i \in [n_0,\ldots,n_{k-1}] \backslash (T_1 \cup \cdots \cup T_{k-1})} |c_i|^2  + \nm{c - c_{[n_{k-1}]} }^2_2
\\
& = \sigma_{t_1+\cdots+t_{k-1}}(c_{[n_0+1,n_{k-1}]})^2_2  + \nm{c - c_{[n_{k-1}]} }^2_2
\\
& \leq \sigma_{s_{k-1}}(c)^2_2 +\nm{c - c_{[n_{k-1}]} }^2_2.
}
Taking the square root, we deduce that $\nm{c_{T_k}}_2 \leq \sigma_{s_{k-1}}(c)_2 + \nm{c-c_{[n_{k-1}]}}_2$ for all $k \geq 2$. This gives the result.
}

\prf{[Proof of Theorem \ref{t:main-res}] 

As in the proof of Theorem \ref{t:basic-CS-thm}, let $E$ be the event that $A$ has the rNSP of order $s$ with constants $\rho = 1/2$ and $\tau = 2 / \sqrt{a_{\phi}}$. Observe that $\bbP(E) \geq 1-\varepsilon/2$ due to \ef{samp-cond-2} and Theorem \ref{t:bounded-riesz-rNSP}.

Suppose that $E$ occurs. Then \ef{step-bounds} holds, where $e$ is as in \ef{bAe-def}. We now consider $\nm{e}_2$. For $k \in \bbN$, let $n_k = 2^{k} n$ and $s_k = 2^k s$. Let $r_1,r_2,\ldots > 0$ be scalars whose values will also be chosen later, and, for $k \in \bbN$, write $F_k$ for the event $\theta_k \leq r_k$, where $\theta_k$ is as in Lemma~\ref{lem:Tail-bound}.
Let $F = \bigcap_{k \in \bbN} F_k$. Then this and the previous lemma give that 
\be{
 \label{F-implies-ebound}
 \text{$F$ occurs} \quad \Rightarrow \quad \nm{e}_2 \leq r_1 \nm{c-c_{[n]}}_2 + \sum^{\infty}_{k=1} r_{k+1} \left(  \sigma_{2^{k} s}(c)_2 + \nm{c - c_{2^{k} n} }_2 \right ).
}
Consider the event $F_k$ and let $\varepsilon_k = \varepsilon/2^{k+1}$. Note that Theorem \ref{t:bounded-riesz-deviation} considers the supremum over $\{ x : \nm{x}_0 \leq s_k,\ \nm{B_k x}_2 = 1\}$, where $B_k = \bbE(A_k^* A_k)^{1/2}$, while $\theta_k$ uses $\nm{x}_2 = 1$. Since
\begin{align*}
\bbE(A_k^*A_k)_{jl} = \frac{1}{m}\sum_i \bbE[\overline{\phi_j(x_i)}\phi_l(x_i)] = \langle \phi_j,\phi_l\rangle_{L^2_\rho},
\end{align*}
the upper Riesz bound gives $\nm{B_k x}_2^2 = x^*\bbE(A_k^*A_k)x = \nms{\sum_j x_j \phi_j}^2_{L^2_\rho} \leq b_\phi\nm{x}_2^2$,
and therefore $\nm{B_k x}_2 \leq \sqrt{b_\phi}\nm{x}_2$. Writing $x = \nm{B_k x}_2 \cdot x/\nm{B_k x}_2$ and applying this gives
\begin{align*}
\theta_k \leq \sqrt{b_{\phi}} \sup_{\nm{B_k x}_2 = 1, \nm{x}_0 \leq s_k} \nm{A_k x}_2.
\end{align*} 
By applying Theorem \ref{t:bounded-riesz-deviation} to the matrix $A_k$, with $s = s_k$, $N = n_k$ we see that
\be{
\label{thetak-cond}
\bbP \left ( \theta_k \geq \sqrt{2b_\phi g_k} \right ) \leq \bbP \left ( \sup_{\nm{B_k x}_2 = 1, \nm{x}_0 \leq s_k} \nm{A_k x}_2 \geq \sqrt{g_k+1} \right )\leq 2 \exp \left ( -  C_1g_k \frac{ a_{\phi} m}{K^2_{\phi} s_k} \right ),
}
for some universal constant $C_1>0$, provided, for some universal constants $C_2,C_3>0$ with $C_3\leq 1$, $g_k$ satisfies
\be{
\label{tk-cond}
\frac{s_k K^2_{\phi} \log^2(2C_2 K^2_\phi s_k / a_{\phi} ) \log(\E n_k) }{a_{\phi} m } + 1 \leq C_3 g_k.
}
Applying the condition \ef{samp-cond-2} and the fact that $s_k = 2^k s$ and $n_k = 2^k n$ we see that
\begin{align}\label{eq:g-value-condition}
\frac{s_k K^2_{\phi} \log^2(2C_2 K^2_\phi s_k / a_{\phi} ) \log(\E n_k) }{a_{\phi} m } & \leq \frac{2^k \log^2(2^{k+1} C_2K^2_{\phi} s / a_{\phi}  ) \log(2^{k}\E n) }{\log^2(2 K^2_{\phi} s/a_{\phi} ) \log(2n) + \log(2/\varepsilon) } \leq c_1 k^3 2^k
\end{align}
for some universal constant $c_1 > 0$.
Hence we now pick $g_k = g_* k^3 2^k$, where $g_* \geq 2 C_3/c_1$ so that \ef{tk-cond} holds. In this case, \ef{thetak-cond} and the condition \ef{samp-cond-2} imply that
\bes{
\bbP \left ( \theta_k \geq \sqrt{2 b_\phi g_*} k^{3/2} 2^{k/2} \right ) \leq 2 \exp \left ( - \frac{g_*}{c_2} k^3 \log(2/\varepsilon) \right ),
}
where $c_2>0$ is a universal constant. Notice that
\bes{
\frac{g_*}{c_2} k^3 \log(2/\varepsilon) \geq (k+2) \log(2) + \log(1/\varepsilon)
}
after potentially increasing $g_*$. Hence
\bes{
\bbP \left (\theta_k \geq r_k \right ) \leq \varepsilon / 2^{k+1} = \varepsilon_k,\quad \text{where }r_k = \sqrt{2 b_\phi g_*} k^{3/2} 2^{k/2},
}
which implies that $\bbP(F^c_k) \leq \varepsilon_k$ with this choice of $r_k$. 

Since $\bbP(F^c) \leq \sum_{k \in \bbN} \bbP(F^c_k) \leq \varepsilon/2$, we deduce that with probability at least $1-\varepsilon/2$, $e$ satisfies
\bes{
\nm{e}_2 \lesssim \sqrt{b_\phi}\left(\tau_n(c)_2 + \sum^{\infty}_{k=1} k^{3/2} 2^{k/2} \left ( \sigma_{2^k s}(c)_2 + \tau_{2^k n}(c)_2 \right )\right).
}
Since $\sigma_s(c)_2$ is nonincreasing in $s$, we have 
\bes{
s(2^k-1) (\sigma_{2^k s}(c)_2)^u \leq \sigma_{s+1}(c)^u_2 + \cdots + \sigma_{2^k s}(c)^u_2 \leq \sum_{j > s} \sigma_j(c)^u_2
}
and therefore
\begin{align}\label{eq:u-small-enough}
\sum^{\infty}_{k=1} k^{3/2} 2^{k/2} \sigma_{2^k s}(c)_2 \leq \left ( \frac1s \sum_{j > s} \sigma_j(c)^u_2 \right )^{\frac1u} \sum^{\infty}_{k=1} \frac{k^{3/2} 2^{k/2}}{(2^k-1)^{\frac1u} } \lesssim_u \left ( \frac1s \sum_{j > s} \sigma_j(c)^u_2 \right )^{\frac1u},
\end{align}
where we recall that $u < 2$. By the same argument, we also have
\bes{
\tau_n(c)_2 + \sum^{\infty}_{k=1} k^3 2^k \tau_{2^k n}(c)_2 \lesssim_v \left ( \frac1n \sum_{j > n/2} \tau_n(c)^v_2 \right )^{\frac1v},
}
as $v < 2$.  Thus, with probability at least $1-\varepsilon/2$, we have
\begin{align}\label{eq:upper-bound-of-e-2-norm}
\nm{e}_{2} \lesssim_{u,v} \sqrt{b_\phi}\left(\left ( \frac1s \sum_{j > s} \sigma_j(c)^u_2 \right )^{\frac1u} + \left ( \frac1n \sum_{j > n/2} \tau_n(c)^v_2 \right )^{\frac1v}\right).
\end{align}
We now substitute this into \ef{l1-bd-temp} and \ef{l2-bd-tmp2} and apply the union bound to deduce that, with probability at least $1-\varepsilon$,
\eas{
\| c_{[n]} - \hat{c} \|_1 
& \lesssim_{u,v}  \sigma_s(c_{[n]})_1 
+  \frac{\sqrt{sb_\phi}}{\sqrt{a_{\phi}}} \left [ \left ( \frac1s \sum_{j > s} \sigma_j(c)^u_2 \right )^{\frac1u} + \left ( \frac1n \sum_{j > n/2} \tau_j(c)^v_2 \right )^{\frac1v} \right ],
\\
\| c_{[n]} - \hat{c} \|_2 
& \lesssim_{u,v}  \frac{\sigma_s(c_{[n]})_1 }{\sqrt{s}}
+  \frac{\sqrt{b_{\phi}}}{\sqrt{a_{\phi}}} \left [ \left ( \frac1s \sum_{j > s} \sigma_j(c)^u_2 \right )^{\frac1u} + \left ( \frac1n \sum_{j > n/2} \tau_j(c)^v_2 \right )^{\frac1v} \right ].
}
Notice that this holds for all $f$ with at least probability $1-\varepsilon$, since the events $E,F$ are independent of $f$. Consider the $L^2_{\rho}$-norm error. Using \ef{riesz-basis} and the bound
\bes{
 \tau_n(c)_2 \leq \left ( \frac{\tau_{n/2+1}(c)^v_2 + \cdots + \tau_{n}(c)^v_2}{n/2} \right )^{\frac1v} \lesssim_v \left ( \frac1n \sum_{j > n/2} \tau_j(c)^v_2 \right )^{\frac1v},
}
we get
\eas{
\nm{f - \hat{f}}_{L^2_{\rho}} & \leq \sqrt{b_{\phi}} \left ( \nm{c_{[n]} - \hat{c}}_2 + \tau_n(c)_2 \right )
\\
& \lesssim_{u,v} \sqrt{b_{\phi}} \left \{  \frac{\sigma_s(c_{[n]})_1 }{\sqrt{s}}
+  \frac{\sqrt{b_{\phi}}}{\sqrt{a_{\phi}}} \left [ \left ( \frac1s \sum_{j > s} \sigma_j(c)^u_2 \right )^{\frac1u} + \left ( \frac1n \sum_{j > n/2} \tau_j(c)^v_2 \right )^{\frac1v} \right ] \right \}.
}
While for the $L^{\infty}_{\rho}$-norm error we use the bound
$
\nm{f - \hat{f}}_{L^{\infty}_{\rho}} \leq  K_{\phi} \left ( \nm{c_{[n]} - \hat{c}}_1 + \tau_n(c)_1 \right )
$
to obtain
\eas{
\nm{f - \hat{f}}_{L^{\infty}_{\rho}} & \lesssim_{u,v} K_{\phi} \left \{ \sigma_s(c_{[n]})_1 
+ \tau_n(c)_1 +  \frac{\sqrt{s}\sqrt{b_{\phi}}}{\sqrt{a_{\phi}}} \left [ \left ( \frac1s \sum_{j > s} \sigma_j(c)^u_2 \right )^{\frac1u} + \left ( \frac1n \sum_{j > n/2} \tau_j(c)^v_2 \right )^{\frac1v} \right ] \right \}
}
with the same probability. We now use \ef{Lp-prod-split}.
}

\section{Application to weighted Wiener-type spaces}\label{s:weighted-wiener}

In this section, we consider the mixed weighted Wiener spaces $S^r_{\theta} \cA(\bbT^d)$ discussed in \S \ref{ss:weighted-wiener-intro}. However, to illustrate the generality of our approach, we work with abstract Wiener-type spaces, from which the main result, Theorem \ref{t:wiener-main}, for $S^r_{\theta} \cA(\bbT^d)$ follows as a special case.

\subsection{Abstract Wiener-type spaces}

Throughout this section, we consider the abstract setting where $(D,\cD,\rho)$ is a probability space and $\{ \phi_i \}_{i \in \bbN}$ is a bounded orthonormal basis of functions defined everywhere. In particular, this need not be the Fourier basis on $\bbT^d$. Note that we could easily consider a Riesz basis, in which case the various error bounds would involve the constants $a_{\phi},b_{\phi}$. For simplicity, we will not do this.
Let $\theta \in (0,\infty]$ and let $w = (w_i)_{i \in \bbN}$ be a sequence of positive weights satisfying
$
w \in \ell^{1/(1-1/\theta)_{+}}(\bbN).
$
Consider the set of functions
\be{
\label{F-weighted-def}
F = F_{w,\theta} = \left \{ f = \sum_{i \in \bbN} c_i \phi_i : \sum_{i \in \bbN} \left ( \frac{|c_i|}{w_i} \right )^{\theta} \leq 1 \right \}.
}
Notice that the series $\sum_{i \in \bbN} c_i \phi_i(x)$ converges absolutely for all $x$, due to the boundedness of the $\phi_i$'s. For $\theta > 1$ this follows from H\"older's inequality and the fact that $w \in \ell^{1/(1-1/\theta)}(\bbN)$. For $\theta \in (0,1]$, it follows from first noting that $\|w\|_{\ell^\infty} < \infty$ and second observing that the condition $\sum_{i \in \bbN} \left(\frac{|c_i|}{w_i}\right)^{\theta} \leq 1$ implies that $\sum_{i \in \bbN} \frac{|c_i|}{w_i} \leq 1$. Hence any $f \in F$ is defined everywhere.

\lem{
\label{F-weighted-sn-term}
Let $f = \sum_{i \in \bbN} c_i \phi_i \in F$. Then the following holds for any $\theta \in [1,\infty]$.
\begin{enumerate}[label = (\roman*)]
\item If $1 \leq p < \theta$ then $\sigma_s(c)_p \leq \sigma_{s}(w)_{\frac{1}{1/p-1/\theta}}.
$
\item If $\theta \leq p \leq \infty$ then
$
\sigma_s(c)_p \leq \inf_{\substack{s_1,s_2 \in \bbN \\ s_1 + s_2 = s}} \left \{ \sigma_{s_1}(w)_{\infty} \cdot s^{1/p-1/\theta}_2 \right \}.
$
\end{enumerate}
Moreover, we also have
\bes{
\sup_{f = \sum_{i} c_i \phi_i \in F} \tau_n(c)_p = \tau_n(w)_{\tilde{p}},\quad \text{where } \tilde{p} = \begin{cases} \frac{1}{1/p-1/\theta} & 1 \leq p < \theta \\ + \infty & \theta \leq p \leq \infty \end{cases} .
}
}

\prf{
Let $\pi : \bbN \rightarrow \bbN$ be a bijection that gives a nonincreasing rearrangement of $w$ and set $d_i = c_i / w_i$. Consider case (i). Then, by H\"older's inequality,
\bes{
\sigma_s(c)_p \leq \left ( \sum_{i > s} | d_{\pi(i)}|^p w^p_{\pi(i)} \right )^{\frac1p} \leq \left ( \sum_{\substack{i > s_1 }} | d_{\pi(i)} |^{\theta} \right )^{\frac{1}{\theta}} \left ( \sum_{\substack{i > s_1 }} |w_{\pi(i)} |^{\frac{p \theta}{\theta-p}} \right )^{\frac{\theta-p}{p\theta}}.
}
For the first term, the definition of $F$ gives $\sum_{\substack{i > s }} | d_{\pi(i)} |^{\theta}  \leq 1$.
For the second term, the fact that $\pi$ is a nonincreasing rearrangement gives 
\bes{
\left ( \sum_{\substack{i > s }} |w_{\pi(i)} |^{\frac{p \theta}{\theta-p}} \right )^{\frac{\theta-p}{p\theta}} \leq \sigma_{s}(w)_{\frac{p \theta}{\theta-p}}.
}
This completes the proof of case (i).

Now consider case (ii). Fix $s_1,s_2 \in \bbN$ with $s_1+s_2 = s$. Let $T \subseteq \{ \pi(i) : i > s_1 \}$, $|T| = s_2$ be the index set of the largest $s_2$ coefficients of $(d_{\pi(i)})_{i > s_1}$ in absolute value. By Stechkin's inequality, we have
\bes{
\sigma_s(c)_p \leq \left ( \sum_{\substack{i > s_1 \\ \pi(i) \notin T }} | d_{\pi(i)} |^p w^{p}_{\pi(i)} \right )^{\frac1p} \leq \sigma_{s_1}(w)_{\infty} \sigma_{s_2}(d)_p \leq \sigma_{s_1}(w)_{\infty} s^{1/p-1/\theta}_2 \nm{d}_{\theta} \leq \sigma_{s_1}(w)_{\infty} s^{1/p-1/\theta}_2.
}
This completes the proof of case (ii).

Consider the second result. For $1 \leq p < \theta$, arguing in a similar way, we have
\bes{
\tau_n(c)_p = \sum_{i > n} |d_i|^p w^p_i \leq \left ( \sum_{i > d} |d_i |^{\theta} \right )^{\frac{1}{\theta}} \left ( \sum_{i > n} w^{\frac{p\theta}{\theta-p}}_i \right )^{\frac{\theta-p}{p\theta}} \leq \tau_n(w)_{\frac{1}{1/p-1/\theta}}.
}
Conversely, for $p \geq \theta$, we have 
\bes{
\tau_n(c)_p \leq \tau_n(c)_{\theta} = \left ( \sum_{i > n} |d_i |^{\theta} w^{\theta}_i \right )^{\frac{1}{\theta}} \leq \tau_n(w)_{\infty},
}
as required. Having shown the desired upper bound, we only need to find an $f \in F$ for which $\tau_n(c)_p \geq \tau_n(w)_{\tilde{p}}$. Suppose that $1 \leq p < \theta$ and define $f = \sum_{i \in \bbN} c_i \phi_i$, where
\bes{
c_1 = \cdots = c_n = 0,\qquad c_i = \frac{w^{\tilde{p}/p}_i}{\tau_n(w)^{\tilde{p}/\theta}_{\tilde{p}}},\quad i > n.
}
Notice that $f \in F$ by construction, and also that
\bes{
\tau_n(c)^p_p = \sum_{i > n}  \frac{w^{\tilde{p}}_i}{\tau_n(w)^{p \tilde{p}/\theta}_{\tilde{p}}} = \tau_n(w)^{\tilde{p} - p \tilde{p}/\theta}_{\tilde{p}} = \tau_n(w)^p_{\tilde{p}},
}
Hence $\tau_n(c)_p = \tau_n(w)_{\tilde{p}}$, as required. Now consider $p \geq \theta$ and let $f = w_{n^*} \phi_{n^*}$, where $n^* > n$ is such that $\tau_{n}(w)_{\infty} = w_{n^*}$. Then $f \in F$ and we trivially have $\tau_n(c)_p = w_{n^*} = \tau_n(w)_{\infty}$. This gives the result.
}

We now state the following lemma, which is a short exercise.

\lem{
\label{lem:sigma-s-c-wiener-alg-decay}
Let $p \in [1,\infty]$ and $w_i \asymp i^{-r} \log^{t}(i+1)$ for some $r > (1-1/\theta)_+$ and $t \in \bbR$. Then $w \in \ell^{1/(1-1/\theta)_{+}}(\bbN)$ and, for any $p \in [1,\infty]$, we have
\bes{
\sigma_s(w)_p \leq \tau_s(w)_p \lesssim_{r,t,p} s^{\frac1p - r} \log^{t}(s+1).
}
}

Next, we establish the main result for the class \ef{F-weighted-def}, which is based on weights of the form of the previous lemma.

\thm{
\label{t:abstract-wiener-main}
Let $\theta \in (0,\infty]$, $r > (1-1/\theta)_{+}$, $t \in \bbR$ and consider the class
\be{
\label{F-r-t-theta}
F = \left \{ f = \sum_{i \in \bbN} c_i \phi_i : \sum_{i \in \bbN} \left ( \frac{i^{r}}{\log^{t}(i+1)} |c_i| \right )^{\theta} \leq 1 \right \}.
}
Let 
\bes{
n = \begin{cases} \lceil s^{(r+1/\theta-1/2)/(r - (1-1/\theta)_+ ) } \rceil & r \leq 1/2 
\\
\max\{ s , \lceil s^{(r+1/\theta-1/2)/r} \rceil \} & r > 1/2
\end{cases},
}
$0 < \varepsilon < 1$, $s \in \bbN$ and $x_1,\ldots,x_m \sim_{\mathrm{i.i.d.}} \rho$, where $m$ satisfies
\bes{
m \geq c \cdot K^2_{\phi} \cdot s \cdot \left ( \log^2(2 K^2_{\phi} s) \log(2n) + \log(2/\varepsilon) \right ).
}
Then the following holds with probability at least $1-\varepsilon$. For any $f \in F$ the approximation $\hat{f} = \cR_{s,n}(f)$ is at most $2s$-sparse and satisfies
\bes{
\nm{f - \hat{f}}_{L^p_\rho} \lesssim_{r,t} K^{1-2/p}_\phi s^{1-1/p-1/\theta-r} \log^{t}(s+1) .
}
}

\prf{

Since $\{ \phi_i \}_{i \in \bbN}$ is a bounded orthonormal basis, $a_\phi = b_\phi = 1$. Let $f = \sum_{i \in \bbN} c_i \phi_i$ and $c = (c_i)_{i \in \bbN}$. We treat the cases $0 < \theta < 1$, $1 \leq \theta < 2$ and $\theta \geq 2$ separately.

\pbk
\textit{Case 1: $0 < \theta < 1$.} Combining Lemmas~\ref{F-weighted-sn-term} and~\ref{lem:sigma-s-c-wiener-alg-decay} we get
\bes{
\sigma_s(c)_p \lesssim \sigma_{s/2}(w)_{\infty} s^{1/p-1/\theta} \lesssim_{r,t} s^{1/p-1/\theta-r} \log^{t}(s+1),\quad \forall p \in [1,\infty]
}
and
\bes{
\tau_n(c)_p \lesssim \tau_n(w)_{\infty} \lesssim_{r,t} n^{-r} \log^t(n+1).
}
Applying Theorem~\ref{t:basic-CS-thm} with $a_\phi = b_\phi = 1$, we have
\bes{
\nm{f - \hat{f}}_{L^p_\rho} \lesssim_{r,t} K^{1-2/p}_\phi \left( s^{1-1/p-1/\theta-r} \log^{t}(s+1) + s^{1/2-1/p} n^{-r} \log^{t}(n+1) \right).
}
We now use the fact that $n = \lceil s^{(r+1/\theta-1/2)/r} \rceil$ (since $(1-1/\theta)_{+} = 0$) in this case to obtain the desired bound.

\pbk
\textit{Case 2: $1 \leq \theta < 2$.} Combining Lemmas~\ref{F-weighted-sn-term} and~\ref{lem:sigma-s-c-wiener-alg-decay} we get
\bes{
\sigma_{s}(c)_1 \lesssim \sigma_{s}(w)_{\frac{1}{1-1/\theta}} \lesssim_{r,t} s^{1-1/\theta-r} \log^{t}(s+1)
}
and
\bes{
\sigma_s(c)_2 \lesssim \sigma_{s/2}(w)_{\infty} s^{1/2-1/\theta} \lesssim_{r,t} s^{1/2-1/\theta-r} \log^t(s+1),
}
as well as
\bes{
\tau_n(c)_1 \leq \tau_n(w)_{\frac{1}{1-1/\theta}} \lesssim_{r,t} n^{1-1/\theta-r} \log^t(n+1)
}
and
\bes{
\tau_n(c)_2 \leq \tau_n(w)_{\infty} \lesssim_{r,t} n^{-r} \log^t(n+1).
}
We now divide into two cases: (a) $1-1/\theta < r \leq 1/2$ and (b) $r > 1/2$.

\pbk
\textit{Case 2(a): $1-1/\theta < r \leq 1/2$.} In this case, we apply Theorem \ref{t:basic-CS-thm} once more to get
\bes{
\nm{f - \hat{f}}_{L^p_\rho} \lesssim_{r,t} K^{1-2/p}_\phi \left( s^{1-1/p-1/\theta-r} \log^{t}(s+1) + s^{1/2-1/p} n^{1-1/\theta-r} \log^{t}(n+1) \right).
}
The result follows after using the fact that $n = \lceil s^{(r+1/\theta-1/2)/(r - 1 + 1/\theta)} \rceil$ in this case.

\pbk
\textit{Case 2(b): $r > 1/2$.} In this case, we shall apply Theorem \ref{t:main-res}. Since $r > 1/2$, there exists a $v = v(r) \in (0,2)$ such that $r v > 1$. Using the above estimate for $\tau_n(c)_2$, we deduce that
\bes{
\left ( \frac1n \sum_{j > n/2} \tau_j(c)^v_2 \right )^{\frac{1}{v}} \lesssim_{r,t} n^{-r} \log^{t}(n+1).
}
Similarly, since $r > 1-1/\theta$, there exists a $u = u(r)$ such that $r-1/\theta -1/2) u > 1$. It follows that
\bes{
\left ( \frac1n \sum_{j > s} \sigma_j(c)^u_2 \right )^{\frac{1}{u}} \lesssim_{r,t} s^{1/2-1/\theta-r} \log^{t}(s+1).
}
Applying Theorem \ref{t:main-res} we now see that
\bes{
\nm{f - \hat{f}}_{L^p_\rho} \lesssim_{r,t} K^{1-2/p}_\phi \left( s^{1-1/p-1/\theta-r} \log^{t}(s+1) + s^{1/2-1/p} n^{-r} \log^{t}(n+1) \right).
}
The result follows after recalling that $n = \lceil s^{(r+1/\theta-1/2)/r} \rceil$ in this case.

\pbk
\textit{Case 3: $\theta \geq 2$.} Combining Lemmas~\ref{F-weighted-sn-term} and~\ref{lem:sigma-s-c-wiener-alg-decay} we get
\bes{
\sigma_{s}(c)_p \lesssim \sigma_{s}(w)_{\frac{1}{1/p-1/\theta}} \lesssim_{r,t} s^{1/p-1/\theta-r} \log^{t}(s+1),\quad p = 1,2,
}
and
\bes{
\tau_n(c)_p \leq \tau_n(w)_{\frac{1}{1/p-1/\theta}} \lesssim_{r,t} n^{1/p-1/\theta-r} \log^t(n+1),\quad p =1,2.
}
Recall that $r > 1-1/\theta$ and therefore $r + 1/\theta - 1/2 > 1/2$. Hence we can find $u = u(r) \in (0,2)$ such that $(r+1/\theta-1/2) u > 1$. We now apply Theorem \ref{t:main-res} (with $v = u$) to obtain
\bes{
\nm{f - \hat{f}}_{L^p_\rho} \lesssim_{r,t} K^{1-2/p}_\phi \left( s^{1-1/p-1/\theta-r} \log^{t}(s+1) + s^{1/2-1/p} n^{1/2-1/\theta-r} \log^{t}(n+1) \right).
}
The result follows after recalling that $n = s$ in this case.
}

\subsection{Application to weighted mixed Wiener spaces}

We now consider the spaces $S^{r}_{\theta} \cA$ introduced in Definition \ref{ss:weighted-wiener-intro}. In particular, we derive Theorem \ref{t:wiener-main} as a corollary of Theorem \ref{t:abstract-wiener-main}.
In order to connect $S^r_{\theta} \cA$ to the setting of Theorem \ref{t:abstract-wiener-main}, we show that the weights of $S^r_{\theta} \cA$, when ordered in nonincreasing order, are of the type considered therein. The following lemma was also used in part of the proof of \cite[Theorem~4.5]{moeller2026best}.

\begin{lemma}
\label{lem:Salpha-r-A-weights}
Let $v = (v_k)_{k \in \mathbb{Z}^d}$ with $v_k = \prod_{i=1}^d (1+|k_i|)^{-r}$, and $\pi : \bbN \rightarrow \bbZ^d$ be a bijection that gives nonincreasing rearrangement of $v$. Then
\begin{align*}
v_{\pi(i)} \lesssim_{d,r} i^{-r} \log^{(d-1)r}(i+1).
\end{align*}
\end{lemma}

\begin{proof}
For any $i \in \mathbb{N}$, since $\pi$ orders $v$ in nonincreasing order, the set $\{k \in \mathbb{Z}^d : v_k \geq v_{\pi(i)}\}$ contains at least $i$ elements. Observe that
\begin{align*}
\{k \in \mathbb{Z}^d : v_k \geq v_{\pi(i)}\} = \left\{ k \in \mathbb{Z}^d : \prod_{i=1}^d (1+|k_i|) \leq v_{\pi(i)}^{-1/r} \right\} = \Lambda^{\mathsf{HC}}_{v_{\pi(i)}^{-1/r}}.
\end{align*}
Since each factor $(1+|k_i|)$ is a positive integer, the product $\prod_{i=1}^d (1+|k_i|)$ is always a positive integer, so $\Lambda^{\mathsf{HC}}_{v_{\pi(i)}^{-1/r}} = \Lambda^{\mathsf{HC}}_{\lfloor v_{\pi(i)}^{-1/r} \rfloor}$. Applying the size estimate at the integer order $\lfloor w_n^{-1/\alpha} \rfloor \in \mathbb{N}$ (see, e.g., \cite[Prop.~A.1]{migliorati2013polynomial}) and using the fact that $\lfloor v_{\pi(i)}^{-1/r} \rfloor \leq v_{\pi(i)}^{-1/r}$, we get
\begin{align*}
i \leq |\Lambda^{\mathsf{HC}}_{\lfloor v_{\pi(i)}^{-1/r} \rfloor}| \leq \lfloor v_{\pi(i)}^{-1/r} \rfloor \log^{d-1}(\E \cdot \lfloor v_{\pi(i)}^{-1/r} \rfloor) \leq v_{\pi(i)}^{-1/r} \log^{d-1}(\E \cdot v_{\pi(i)}^{-1/r}).
\end{align*}
This gives $v_{\pi(i)}^{-1/r} \gtrsim_d i / \log^{d-1}(i+1)$, where $c_d$ is a constant dependent on $d$, and therefore $v_{\pi(i)} \lesssim_{d,r} i^{-r} \log^{(d-1)r}(i+1)$, as required.
\end{proof}

\prf{[Proof of Theorem \ref{t:wiener-main}]
Let $v$ and $\pi$ be as in the previous lemma, and define the weights
$w_i = i^{-r} \log^{(d-1) r}(i+1)$. Then the previous lemma implies that
\bes{
f \in S^{r}_{\theta} \cA(\bbT^d) \quad \Longrightarrow  \quad \tilde{f} : = \frac{f}{C_{d,r,\theta} \nm{f}_{S^r_{\theta} \cA} } \in F_{w,\theta},
}
where $C_{d,r,\theta} > 0$ depends on $d$, $r$ and $\theta$ only. Notice that the sparse recovery procedure $\cR_{s,n}$ satisfies $\cR_{s,n}(C f) = C \cR_{s,n}(f)$ for any $C>0$. Hence we may apply Theorem \ref{t:abstract-wiener-main} with $t = (d-1)r$ to $\tilde{f}$ and $K_{\phi} = 1$ to deduce that
\bes{
\nm{f - \hat{f}}_{L^p} \lesssim_{r,d} s^{1-1/p-1/\theta-r} \log^{(d-1)r}(s+1) \nm{f}_{\cS^{r}_{\theta} \cA},
}
as required.
}

\subsection{Comparison to the sampling widths of $S^{r}_{\theta} \cA(\bbT^d)$}\label{ss:wiener-samp-widths}

As in \cite{moeller2026best}, we now consider the nonlinear sampling width $\varrho_m$ \cite[Defn.\ 3.4]{moeller2026best}.

\begin{definition}
Let $F$ be a (quasi-)normed space of functions $D \rightarrow \bbC$, where function evaluations are continuous, which is continuously embedded into a Banach space $Y$. The $m$th \textit{(nonlinear) sampling width} is
\begin{align*}
\varrho_m(F)_Y = \inf_{x_1,\ldots,x_m \in D} \inf_{R\colon\mathbb{C}^m \to Y} \sup_{\|f\|_F \leq 1} \|f - R(f(x_1),\ldots,f(x_m))\|_Y.
\end{align*}
\end{definition}

We consider the case $Y = L^2(\bbT^d)$ and $F = S^{r}_{\theta} \cA(\bbT^d)$. Applying Theorem \ref{t:wiener-main} with
\be{
\label{s-choice-m}
s = c_{d,r,\theta} \frac{m}{\log^3(m+1)},
}
for some suitable constant $c_{d,r,\theta}$, we deduce the upper bound
\be{
\label{rhom-our-rate}
\varrho_m(S^{r}_{\theta} \cA(\bbT^d))_{L^p} \lesssim_{d,r,\theta} m^{1-1/p-1/\theta-r} \log^{(d-1)r+3(r+1/p+1/\theta-1)}(m+1).
}
As discussed in \cite{moeller2026best}, when $p = 2$ and $\theta\in(0,2]$, on has the lower bound (see \cite[Lem.\ B.1]{jahn2023sampling})
\bes{
\varrho_m(S^{r}_{\theta} \cA(\bbT^d))_{L^p} \gtrsim_{d,r,\theta} m^{1/2-1/\theta-r} \log(m+1)^{(d-1)r}.
}
Hence the upper bound \ef{rhom-our-rate} is sharp up to the additional factor $5/2+r+1/\theta$ appearing in the exponent of the log term. Moreover, the recovery procedure uses a truncation set of size
\be{
\label{wiener-truncation-size-us}
| \Lambda | = n = \begin{cases} \lceil s^{(r+1/\theta-1/2)/(r - (1-1/\theta)_+ ) } \rceil & r \leq 1/2 
\\
\max \{ s , \lceil s^{(r+1/\theta-1/2)/r} \rceil \} & r > 1/2
\end{cases},\qquad \text{where $s$ is as in \ef{s-choice-m}.}
}

\subsection{Comparison to the results of \cite{moeller2026best}}\label{ss:moller-comparison}

We now describe how Theorem \ref{t:wiener-main} improves on the results of \cite{moeller2026best}. In \cite[Cor.\ 6.2]{moeller2026best} the authors establish a bound for the sampling width of the form
\bes{
\varrho_{m}(S^r_{\theta} \cA)_{L^p} \lesssim_{d,r,\theta} m^{1-1/p-1/\theta-r} \log^{(d-1)r+3(r+1/p+1/\theta-1)}(m+1) .
}
Notice that this is precisely the same rate as in \ef{rhom-our-rate}.
This is done using a sparse recovery procedure (also involving the SR-LASSO) with truncation set
\be{
\label{moller-truncation}
\Lambda = [-M,M]^d \cap \bbZ^d.
}
Let $\cT_M = \left \{ g = \sum_{k \in [-M,M]^d \cap \bbZ^d} c_k \phi_k \right \}$ be the set of multivariate trigonometric polynomials with coefficients limited to the cube $[-M,M]^d$ and define 
\bes{
E_{[-M,M]^d}(S^r_{\theta}\cA)_{L^{\infty}} = \sup_{\nm{f}_{S^r_{\theta} \cA} \leq 1} E_{[-M,M]^d}(f),\qquad E_{[-M,M]^d}(f) = \inf_{g \in \cT_M} \nm{f - g}_{L^{\infty}},
}
In \cite[Proof of Cor.\ 6.2]{moeller2026best}, the authors first establish an error bound of the form 
\be{
\label{rhom-moller-bd}
\varrho_{m}(S^r_{\theta} \cA)_{L^q} \lesssim_{d,r,\theta} s^{1/2-1/p} \left ( s^{1/2-1/\theta -r} \log(s+1)^{(d-1) r} + E_{[-M,M]^d}(S^r_{\theta} \cA)_{L^{\infty}} \right ),
}
where $s$ is as in \ef{s-choice-m} (see \cite[Eqn.\ (3.1)]{moeller2026best}). The following result derives the precise scaling of the term $E_{[-M,M]^d}(S^r_{\theta} \cA)_{L^{\infty}}$.

\prop{
\label{prop:Linf-cube-rate}
The term $E_{[-M,M]^d}(S^r_{\theta} \cA)_{L^{\infty}}$ satisfies
\bes{
E_{[-M,M]^d}(S^r_{\theta} \cA)_{L^{\infty}} \asymp_{d,r,\theta} M^{(1-1/\theta)_+-r}.
}
}

Using this proposition, balancing terms in \ef{rhom-moller-bd} results in the choice
\bes{
M \asymp s^{(r+1/\theta-1/2)/(r-(1-1/\theta)_+)} .
}
Further, Proposition \ref{prop:Linf-cube-rate} shows that this is the best possible choice of $M$. However, this implies that the truncation set \ef{moller-truncation} satisfies
\be{
\label{wiener-truncation-size-them}
n = |\Lambda| \asymp s^{d(r+1/\theta-1/2)/(r-(1-1/\theta)_+)}.
}
where $s$ is as in \ef{s-choice-m}. Upon comparison with \ef{wiener-truncation-size-us}, we see that this size is much larger. In general the exponent of $s$ is $d$ times larger, meaning that the truncation set if \cite{moeller2026best} suffers from a severe curse of dimensionality, whereas ours does not. Moreover, when $\theta > 2$, the cardinality \ef{wiener-truncation-size-them} blows up as $r \rightarrow (1-1/\theta)_+$, while the size of our set \ef{wiener-truncation-size-us} remains bounded.

\rem{[The case $\theta > 2$]
If $\theta \geq 2$ then the truncation set $\Lambda$ in our recovery map satisfies $|\Lambda| = n = s$ due to \ef{wiener-truncation-size-us}. This means that the underlying problem being solved is no longer a sparse recovery problem, since $s = n$. This is indicative of the fact that linear algorithms can achieve the near-optimal rates for $\theta \geq 2$.

For $\theta \in [2,\infty]$ with $r>1-1/\theta$, a linear sampling algorithm achieves the rate
\begin{align*}
m^{-\left(r - 1 + \frac{1}{\theta} + \frac{1}{p}\right)} (\log m)^{(d-1)\left(r - \frac{1}{2} + \frac{1}{p}\right)}, \qquad p \in [2,\infty].
\end{align*}
This follows from \cite[Thm.\ 5.3, Rem.\ 5.4(ii)]{kolomoitsev2023sparse} with $\beta = 0$ and $N = 0$. The algorithm is the quasi-interpolation operator $P^Q_{n,0}$ of \cite{kolomoitsev2023sparse} with $Q = I$, where $I$ is the Dirichlet-type Lagrange interpolation operator of \cite[Example~2.3(i)]{kolomoitsev2023sparse}, with $m \asymp 2^n n^{d-1}$. The polynomial factor in $m$ is the same as in \ef{rho-m-upper-bound}. For $p = 2$, the logarithmic factors above are smaller than those in \ef{rho-m-upper-bound}, where the difference is due to the fact that \ef{rho-m-upper-bound} was derived via a compressed sensing formulation requiring $m \gtrsim s \log^3 s$ samples. For $p \in (2,\infty]$, the rate above is better than \ef{rho-m-upper-bound} even accounting for this, where this difference too is in the power of the logarithmic factors.

Notice that our reconstruction method is nonlinear. However, it remains valid in the `linear-is-sufficient' regime $\theta \geq 2$ and, moreover, in this regime the truncation set shrinks to the minimal size $n = s$. This does not occur in the case of \cite{moeller2026best}. In particular, \ef{wiener-truncation-size-them} blows up as $r \rightarrow (1-1/\theta)^+$, as noted above.

We remark in passing that in the case $\theta \geq 2$ one could use the larger value $s = c_{d,r,\theta} m / \log(m+1)$ rather than that given by \ef{s-choice-m}. This arises because $s = n$, meaning that the required properties (i.e., Theorems \ref{t:bounded-riesz-rNSP} and \ref{t:bounded-riesz-deviation}) reduce to estimates of the maximal and minimal singular values of $A$. These can be estimated more sharply using simpler matrix Chernoff bounds (see, e.g., \cite{tropp2012user-friendly}), rather than the more involved chaining arguments used in the proofs of Theorems \ref{t:bounded-riesz-rNSP} and \ref{t:bounded-riesz-deviation}. For succinctness we omit this derivation.
}

\prf{[Proof of Proposition \ref{prop:Linf-cube-rate}]
Let $f = \sum_{k \in \bbZ^k} \hat{f}_k \phi_k$ with $\nm{f}_{S^r_{\theta} \cA } \leq 1$. Then
$
E_{[-M,M]^d}(f)_{L^\infty} \leq \sum_{k \notin [-M,M]^d} |\hat{f}_k|.
$
Consider the weights $v_k = \prod^{d}_{i=1} (1+|k_i|)^{-r}$. By Lemma \ref{F-weighted-sn-term}, we see that
\bes{
\sum_{k \notin [-M,M]^d} |\hat{f}_k| \leq \left ( \sum_{k \notin [-M,M]^d} v^{\frac{1}{1-1/\theta}}_k \right )^{1-1/\theta} ,\qquad 1 < \theta \leq \infty,
}
and 
\bes{
\sum_{k \notin [-M,M]^d} |\hat{f}_k| \leq \sup_{k \notin [-M,M]^d} v_k,\qquad 0 < \theta \leq 1.
}
In the former case, notice that
\eas{
\sum_{k \notin [-M,M]^d} v^{\frac{1}{1-1/\theta}}_k & \leq d \left ( \sum_{\substack{k \in \bbZ \\ |k| > M}} (1+|k_1|)^{-r/(1-1/\theta)} \right ) \left ( \sum_{k \in \bbZ} (1+|k|)^{-r/(1-1/\theta)}  \right )^{d-1} 
\\
& \lesssim_{d,r,\theta} M^{1-r/(1-1/\theta)},
}
which gives
\bes{
E_{[-M,M]^d}(f)_{L^{\infty}} \lesssim_{d,r,\theta} M^{(1-1/\theta)-r}, \qquad 1 < \theta \leq \infty.
}
In the latter case, we straightforwardly see that
\bes{
E_{[-M,M]^d}(f)_{L^{\infty}} \lesssim_{d,r,\theta} M^{-r},\qquad 0 < \theta \leq 1.
}
This yields the desired upper bound.

We now establish the lower bound. Consider the case $1 < \theta \leq \infty$. The first step is to define a suitable function $f = f_M$.
Consider the shifted Dirichlet kernel function (see \cite[\S 1.2.1]{temlyakov2018multivariate})
\begin{align*}
f_M(x) = C_M \sum_{k_1 = M+1}^{2M} \psi_{(k_1, 0, \ldots, 0)}(x) = C_M \sum_{k_1=M+1}^{2M} \E^{2\pi \I k_1 x_1},
\end{align*}
which depends on $x_1$ only. For $1 \leq \theta < \infty$, we have
\begin{align*}
\|f_M\|_{S^r_{\theta} \cA}^{\theta} = C_M^{\theta}  \sum_{k_1=M+1}^{2M} (1+k_1)^{r \theta} \asymp_{r,\theta} C_M^{\theta} M^{r \theta+1},
\end{align*}
and therefore $\nm{f_M}_{S^r_{\theta} \cA} \asymp_{r,\theta} C_M M^{r+1/\theta}$. For $\theta = \infty$ we have 
\bes{
\|f_M\|_{S^r_{\infty} \cA} = C_M \sup_{k_1 \in (M, 2M]} (1+k_1)^r \asymp_r C_M M^r,
}
Therefore, we now pick $C_M \asymp_{r,\theta} M^{-r-1/\theta}$ so that $\nm{f_M}_{S^{r}_{\theta} \cA} \leq 1$.

We now lower bound $\nm{f_M - g}_{L^{\infty}}$ for any $g \in \cT_M$. To do this, let $K^N$ be the Fej\'{e}r kernel of order $N$ on $\mathbb{T}$ (see \cite[Section 1.2.2]{temlyakov2018multivariate}), defined by
$K^N(x) = \sum_{|j| \leq N} \left(1 - \frac{|j|}{N+1}\right) \E^{2\pi \I j x},
$
which satisfies $\|K^N\|_{L^1(\mathbb{T})} \asymp 1$. Define the shifted Fej\'{e}r kernel on $\bbT^d$ by
$
h(x) = \E^{2\pi \I (\lfloor 3M/2 \rfloor + 1) x_1} K^{\lfloor M/2 \rfloor}(x_1),
$
which depends only on $x_1$, satisfies $\|h\|_{L^1(\mathbb{T}^d)} \asymp 1$, and has Fourier support only in the $x_1$ direction, on frequencies $j_1 \in [M+1, 2M]$. Since this support lies outside $[-M, M]^d$, we have $\int_{\mathbb{T}^d} h(x) g(x)\, \D x = 0$ for all $g \in \cT_M$. Hence, for any such $g$, we have
\begin{align*}
\|f_M - g\|_{L^\infty} \geq \frac{\left|\int_{\mathbb{T}^d} h(x)(f_M(x) - g(x))\, \D x\right|}{\|h\|_{L^1(\mathbb{T}^d)}} \gtrsim \left|\int_{\mathbb{T}} h(x_1)f_M(x_1)\, \D x_1\right|.
\end{align*}
Further, we have
\begin{align*}
\int_{\mathbb{T}} h(x_1)f_M(x_1)\, \D x_1
= C_M \sum_{k_1=M+1}^{2M+1} \hat{K}^{\lfloor M/2 \rfloor}_{k_1 - \lfloor 3M/2 \rfloor} \asymp C_M M.
\end{align*}
Therefore,
\begin{align*}
E_{[-M,M]^d}(f_M)_{L^\infty} \gtrsim C_M M \asymp_{r,\theta} M^{-(\alpha+1/r)} \cdot M = M^{1-r-1/\theta}.
\end{align*}
This gives the desired result for $1 < \theta \leq \infty$.

We now consider the case $0 < \theta \leq 1$. Define the function $f_M(x) =  (M+2)^{-r} e^{2\pi \I (M+1) x_1}$ and observe that $\nm{f_M}_{S^{r}_{\theta} \cA} = 1$. For any $g \in \cT_M$, since $f_M$ and $g$ have disjoint Fourier supports, we have 
\begin{align*}
\|f_M - g\|_{L^2}^2 = \|f_M\|_{L^2}^2 + \|g\|_{L^2}^2 \geq \|f_M\|_{L^2}^2 = (M+2)^{-2r}.
\end{align*}
It follows that
\bes{
E_{[-M,M]^d}(S^r_{\theta} \cA)_{L^\infty} \geq E_{[-M,M]^d}(f_M)_{L^2}  \geq (M+2)^{-r}. \gtrsim_r M^{-r},
}
as required.
}

\section{Application to anisotropic Sobolev spaces}\label{s:anisotropic-sobolev}

We now consider the spaces \ef{Halphamix}.

\subsection{Proof of Theorem \ref{t:aniso-sob-main}.}

Since the Fourier basis is indexed over $\bbZ^d$ and our main results are formulated for bases indexed over $\bbN$, our first step is to re-index the Fourier basis. For $r \in \bbN$, let 
\bes{
\Lambda^{\mathsf{HC}}_r = \left \{ n = (k_1,\ldots,k_d) \in \bbZ^d : \prod^{d}_{j=1} (1+|k_j|) \leq r \right \}
}
be the hyperbolic cross index set of order $r$. Write $M_r = |\Lambda^{\mathsf{HC}}_r |$ and notice that $M_1 = 1$. Now let $\pi : \bbN \rightarrow \bbZ^d$ be a bijection such that
\bes{
\left \{ \pi(1),\ldots,\pi(M_r) \right \} = \Lambda^{\mathsf{HC}}_r,\quad \forall r \in \bbN
}
and define $\phi_i = \psi_{\pi(i)}$, $\forall i \in \bbN$. For later use, we now recall that $M_r$ satisfies the bound
\be{
\label{HC-bound}
\frac{1}{(d-1)!} \frac{r (\log r)^d}{\log r + d} \leq M_r \leq r \log^{d-1}(\E r), \quad \forall r , d \in \bbN,
}
(see, e.g., \cite[\S B.2]{adcock2022sparse}). We now require the following lemma.

\lem{
Let $f \in H^{\alpha}_{\mathsf{mix}}(\bbT^d)$, $\alpha > 1/2$, and write $c = (c_i)_{i \in \bbN}$ for its vector of Fourier coefficients, i.e., $c_i = \hat{f}_{\pi(i)}$. Then
\bes{
\sigma_s(c)_q \lesssim_{d,\alpha,q} s^{\frac1q-\frac12-h(\alpha)} (\log s)^{h(\alpha)(p(\alpha)-1)} \nm{f}_{H^{\alpha}_{\mathsf{mix}} },
}
for any $s \in \bbN$ and $q \in [1,\infty]$, and, for any $q \in [1,2]$ and $n \in \bbN$, 
\bes{
\tau_n(c)_q \lesssim_{d,\alpha,q} r^{\frac1q-\frac12-h(\alpha)} \log(r)^{(d-1)(\frac1q-\frac12)} \nm{f}_{H^{\alpha}_{\mathsf{mix}} },
}
where $r \in \bbN$ is the largest integer such that $n \geq M_r$. Moreover, we also have
\bes{
\tau_n(c)_q \lesssim_{d,\alpha,q} n^{\frac1q-\frac12-h(\alpha)} \log(n)^{(d-1)h(\alpha)} \nm{f}_{H^{\alpha}_{\mathsf{mix}} }.
}
}
\prf{
The first result follows directly from \cite[Thm.\ 2.7]{adcock2026universal}. We now consider the second result.  From the definition of $r$, we have $M_r \leq n < M_{r+1}$. We now have
\eas{
\tau_n(c)^q_q & \leq \sum_{n \notin \Lambda^{\mathsf{HC}}_{r} } | \hat{f}_n |^q = \sum_{n \notin \Lambda^{\mathsf{HC}}_{r} } | \hat{f}_n |^q \frac{\prod^{d}_{j=1} (1+|k_j|)^{qh(\alpha)}}{\prod^{d}_{j=1} (1+|k_j|)^{q h(\alpha)}}  
}
By H\"older's inequality, we obtain
\bes{
\tau_n(c)^q_q \leq \left ( \sum_{n \notin \Lambda^{\mathsf{HC}}_{r} } | \hat{f}_n |^2 \prod^{d}_{j=1} (1+|k_j|)^{2h(\alpha)} \right )^{\frac{q}{2}} \left ( \sum_{n \notin \Lambda^{\mathsf{HC}}_{r} } \prod^{d}_{j=1} (1+|k_j|)^{-2qh(\alpha)/(2-q)} \right )^{1-\frac{q}{2}} = I_1 \cdot I_2.
}
For $I_1$, we use the fact that $h(\alpha) \leq \alpha_j$, $\forall j \in [d]$, to get
$
I_1 \leq \nm{f}^{q}_{H^{\alpha}_{\mathsf{mix}}}.
$
For $I_2$, we use, e.g., \cite[Thm.\ 2.30]{adcock2011modified} to get
\bes{
I_2 \lesssim_{d,\alpha,q} \left ( r^{1-2qh(\alpha)/(2-q)} \log(r)^{d-1} \right )^{1-\frac{q}{2}}.
}
Taking the $q$th root, we deduce that
\bes{
\tau_n(c)_q \lesssim \nm{f}^{q}_{H^{\alpha}_{\mathsf{mix}}} \lesssim_{d,\alpha,q} r^{\frac1q-\frac12-h(\alpha)} \log(r)^{(d-1)(\frac1q-\frac12)},
}
as required. 
For the final result, we use \ef{HC-bound} and the fact that $M_r \geq r$ to deduce that
\bes{
r \leq n < (r+1) \log^{d-1}(\E(r+1)).
}
It follows that $n \geq r \gtrsim_d n / \log^{d-1}(2n)$. We now apply the previous result.
}

\lem{
\label{lem:aniso-Sob-main-err-ubs}
Let $s,n \in \bbN$, $s , n \geq 2$, and $f \in H^{\alpha}_{\mathsf{mix}}(\bbT^d)$, where $\alpha > 1/2$, and consider the right-hand sides of \ef{main-err-bd-l2} and \ef{main-err-bd-lp}. Then there is a choice of $u,v \in (0,2)$ depending on $\alpha$, $d$ only such that the right-hand side of \ef{main-err-bd-l2} is bounded, up to a constant depending on $\alpha$ and $d$ only, by
\bes{
E_2 : = \left ( s^{-h(\alpha)} (\log s)^{h(\alpha)(p(\alpha)-1)} + n^{-h(\alpha)} \log(n)^{(d-1) h(\alpha)} \right ) \nm{f}_{H^{\alpha}_{\mathsf{mix}} }
}
and the right-hand side of \ef{main-err-bd-lp} is bounded, up to a constant depending on $\alpha$, $p$ and $d$ only, by
\eas{
E_p : = \Bigg ( & s^{\frac12-\frac1p - h(\alpha)} (\log s)^{h(\alpha)(p(\alpha)-1)} + s^{-\frac1p} n^{\frac12-h(\alpha)} (\log n)^{(d-1)h(\alpha)}
\\
& + s^{\frac12-\frac1p} n^{-h(\alpha)} \log(n)^{(d-1) h(\alpha)} \Bigg ) \nm{f}_{H^{\alpha}_{\mathsf{mix}}}.
}
In particular,
\bes{
n \gtrsim_d s (\log s)^{d-1} \quad \Rightarrow \quad E_2 \lesssim_d s^{-h(\alpha)} (\log s)^{h(\alpha)(p(\alpha)-1)} \nm{f}_{H^{\alpha}_{\mathsf{mix}}}
}
and
\bes{
n \gtrsim_d s (\log s)^{\frac{(d-1) h(\alpha) }{h(\alpha)-1/2}} \quad \Rightarrow \quad E_p \lesssim_d s^{\frac12-\frac1p-h(\alpha)} (\log s)^{h(\alpha)(p(\alpha)-1)} \nm{f}_{H^{\alpha}_{\mathsf{mix}}}.
}
}
\prf{
Recall that $a_{\phi} = b_{\phi} = K_{\phi} = 1$, since the Fourier basis is orthonormal and uniformly bounded by one. By the previous lemma, the right-hand side of \ef{main-err-bd-l2} is bounded, up to a constant depending on $\alpha$, $d$ only, by
\eas{
\widetilde{E}_2 : = &~ s^{-h(\alpha)} (\log s)^{h(\alpha)(p(\alpha)-1)} +  \left ( \frac1s \sum_{j > s} \left ( j^{-h(\alpha)} (\log j)^{h(\alpha)(p(\alpha)-1)} \right )^u \right )^{\frac1u} 
\\
& + \left ( \frac1n \sum_{j > n/2} \left(j^{-h(\alpha)} \log(j)^{(d-1) h(\alpha)} \right )^{v} \right )^{\frac1v} 
}
Since $h(\alpha) > 1/2$, there exists values of $u,v$ depending on $\alpha$ only such that both sums converge, and give
\bes{
\widetilde{E}_2 \lesssim_{d,\alpha} s^{-h(\alpha)} (\log s)^{h(\alpha)(p(\alpha)-1)} + n^{-h(\alpha)} \log(n)^{(d-1) h(\alpha)},
}
as required. The argument for the right-hand side of \ef{main-err-bd-lp} is similar. The final two results follow simply by using the inequality for $n$.
}

\prf{[Proof of Theorem \ref{t:aniso-sob-main}]

Let $\Lambda = \Lambda^{\mathsf{HC}}_s$.
The basis $\{ \phi_i \}_{i \in \bbN}$ is a bounded orthonormal system with $a_{\phi}  = b_{\phi} = K_{\phi} = 1$. Using this and \ef{HC-bound}, we see that \ef{samp-cond-2} is implied by \ef{m-cond-HC}. It follows from Theorem \ref{t:main-res} that, with probability at least $1-\varepsilon$, the approximation $\hat{f}$ satisfies \ef{main-err-bd-l2} for any $f \in L^2(\bbT^d)$ that is defined everywhere. Now suppose that $f \in H^{\alpha}_{\mathsf{mix}}(\bbT^d)$ for some $h(\alpha) > 1/2$. The latter condition implies that $f \in C(\bbT^d)$ \cite[Rem.\ 3.3]{adcock2026universal}. Therefore $\hat{f}$ satisfies \ef{main-err-bd-l2}, where $c = (c_i)_{i \in \bbN}$ is its vector of Fourier coefficients, i.e., $c_i = \hat{f}_{\pi(i)}$. Using Lemma \ref{lem:aniso-Sob-main-err-ubs} and the fact that $n = |\Lambda^{\mathsf{HC}}_s| \gtrsim_d s (\log s)^{d-1}$ by \ef{HC-bound}, we deduce that
\bes{
\nm{f - \hat{f}}_{L^2_{\rho}} \lesssim_{d,\alpha} s^{-h(\alpha)} (\log s)^{h(\alpha)(p(\alpha)-1)} \nm{f}_{H^{\alpha}_{\mathsf{mix}} }.
}
This gives the desired error bound. The estimate \ef{aniso-sob-n-new} follows from \ef{HC-bound}.
}

\subsection{Comparison to the results of \cite{adcock2026universal}}\label{ss:Halphamix-comparison}

In \cite[Thm.\ 3.1]{adcock2026universal} the authors establish a similar result using a similar sparse recovery algorithm. Inspecting the proof, this result uses a hyperbolic cross index set $\Lambda = \Lambda^{\mathsf{HC}}_r$, where $r = \lceil s^{u(s)} \rceil$ and $u(s)$ is some fixed, but arbitrary, increasing function of $s$ with $u(s) \rightarrow \infty$ as $s \rightarrow \infty$. Due to \ef{HC-bound}, this means that 
\bes{
n = |\Lambda| \gtrsim_{d} s^{u(s)} (u(s) \log(s+1))^{d-1}.
}
In particular, $n$ grows superalgebraically with $s$ as $s \rightarrow \infty$. Furthermore, in \cite[Thm.\ 3.1]{adcock2026universal} the term $s$ is chosen so that
\bes{
m \geq c_d \cdot s \cdot \left ( \log^3(2s) \cdot u(s) + \log(1/\varepsilon) \right ).
}
Our result improves this scaling by removing the factor $u(s)$.

\section{Conclusion}\label{s:conclusion}

A standard approach to function approximation from random samples is to truncate the basis expansion at index $n$, form the measurement matrix from i.i.d.\ samples, and apply compressed sensing to recover a sparse coefficient vector. Truncating introduces the discrete error term $\nm{e}_2$ (defined in \ef{bAe-def}), which in current approaches is bounded by $\nm{f - f_n}_{L^\infty_\rho}$(which is further bounded by $\tau_n(c)_1$). This is a worst-case bound that ignores the i.i.d.\ structure of the sample points, and with high probability i.i.d.\ points do not realise the pointwise maximum, so the standard bound is overly pessimistic. Its slow decay in $n$ forces one to use large truncation sets, yielding large measurement matrices and high computational cost. Since $\nm{e}_2^2 = \frac{1}{m}\sum_{i=1}^m |f(x_i) - f_n(x_i)|^2$ is an empirical sum-of-squares of the tail, at i.i.d.\ sample points it should track $\nm{f - f_n}_{L^2_\rho}^2$ with high probability rather than the $L^\infty$ norm.

Our alternate approach to bounding $\nm{e}_2$, inspired by the upper bounds in the form of scaled tail sum of approximation numbers and variants in works such as \cite{krieg2021function,krieg2021function2,nagel2022new,dolbeault2023sharp,kammerer2021worst}, yields depending on terms of the form $(\frac{1}{s}\sum_{j>s}\sigma_j(c)_2^u)^{1/u}$ and $(\frac{1}{n}\sum_{j>n/2}\tau_j(c)_2^v)^{1/v}$, for any $0 < u,v < 2$. For functions with sufficient smoothness, these quantities decay no worse than the corresponding terms $\sigma_s(c)_2$ and $\tau_n(c)_2$ asymptotically. In the applications to weighted Wiener spaces and anisotropic Sobolev spaces, this translates to a truncation set that is essentially linear in $s$, giving smaller matrices and lower computational cost. 

There are several avenues for future work. First, our results require $\ell^u$- and $\ell^v$-summability of the terms $\sigma_j(c)_2$ and $\tau_j(c)_2$ for some $0 < u,v < 2$, which may not hold in all cases. It would be interesting to see if this could be removed. Second, our results only consider sparse approximation using bounded Riesz bases (i.e., those satisfying \ef{Kphi-def}). For unbounded bases---a key example being the Legendre polynomials---the concept of \textit{weighted} sparsity was introduced \cite{rauhut2016interpolation}, and applied to the approximation of classes of holomorphic functions \cite{adcock2024efficient,adcock2025near,adcock2024optimal,adcock2025optimal}. We believe our main result may extend to weighted sparsity, and intend to explore this in future work.

\section*{Acknowledgements}
BA and SB acknowledge support from the Natural Sciences and Engineering Research
Council of Canada (NSERC) through grants RGPIN/2026-04531 and RGPIN/2020-06766, respectively. BA, SB \& AG acknowledge the support
of FRQ (Fonds de recherche du Qu\'ebec) – Nature et Technologies through grant 359708.

\bibliographystyle{abbrv}
\small
\bibliography{Sparse-widths-opt-bib}

\end{document}